\documentclass[10pt]{article}
\usepackage{graphicx} % Required for inserting images
\usepackage[utf8]{inputenc}
\usepackage{amsmath}
\usepackage{combelow}
\usepackage[pdftex,pdfpagelabels]{hyperref}
\usepackage{combelow}
\usepackage{amssymb}
\usepackage{amsthm}
\usepackage{enumerate}
\usepackage{fullpage}
\usepackage{url}
\usepackage{float}
\usepackage{verbatim}
\usepackage{afterpage}
\usepackage[backend=biber, style=numeric, giveninits=true, maxnames=4]{biblatex}
\theoremstyle{plain}
\newtheorem{thm}{Theorem}[section]
\newtheorem{defi}[thm]{Definition}
\newtheorem{lem}[thm]{Lemma}

\newtheorem{prop}[thm]{Proposition}
\newtheorem{rmk}[thm]{Remark}

\newcommand{\R}{\mathbb{R}}
\newcommand{\N}{\mathbb{N}}

\newcommand{\zer}{{\rm zer}}

\DeclareMathOperator{\Id}{Id}

\DeclareMathOperator{\distance}{dist}

\date{}
\title{On a Stochastic Differential Equation with Correction Term Governed by a Monotone and Lipschitz Continuous Operator}
\author{Radu Ioan Bo\cb{t} \footnote{Faculty of Mathematics, University of Vienna, Oskar-Morgenstern-Platz 1, 1090 Vienna, Austria, e-mail: \url{radu.bot@univie.ac.at}. Research partially supported by FWF (Austrian Science Fund), projects W 1260 and P 34922-N, and by a grant of the Romanian Ministry of Research, Innovation and Digitization, CNCS -
UEFISCDI, project number PN-III-P1-1.1-TE-2021-0138, within PNCDI III.}  \and Chiara Schindler \footnote{Faculty of Mathematics, University of Vienna, Oskar-Morgenstern-Platz 1, 1090 Vienna, Austria, e-mail: \url{chiara.schindler@univie.ac.at}.}}

\begin{document}

\maketitle

\begin{abstract}
    In our pursuit of finding a zero for a monotone and Lipschitz continuous operator $M : \R^n \rightarrow \R^n$ amidst noisy evaluations, we explore an associated differential equation within a stochastic framework, incorporating a correction term. We present a result establishing the existence and uniqueness of solutions for the stochastic differential equations under examination. Additionally, assuming that the diffusion term is square-integrable, we demonstrate the almost sure convergence of the trajectory process $X(t)$ to a zero of $M$ and of $\|M(X(t))\|$ to $0$ as $t \rightarrow +\infty$. Furthermore, we provide ergodic upper bounds and ergodic convergence rates in expectation for $\|M(X(t))\|^2$ and $\langle M(X(t), X(t)-x^*\rangle$, where $x^*$ is an arbitrary zero of the monotone operator. Subsequently, we apply these findings to a minimax problem. Finally, we analyze two temporal discretizations of the continuous-time models, resulting in stochastic variants of the Optimistic Gradient Descent Ascent  and Extragradient methods, respectively, and assess their convergence properties.
\end{abstract}

\vspace{1ex}


\section{Introduction}\label{sec1}

For a monotone and $L$-Lipschitz continuous operator $M:\R^n \rightarrow \R^n$, we examine the equation
\begin{align}\label{eq:find-zeros}
    M(x) = 0.
\end{align}
We assume that the set of solutions of \eqref{eq:find-zeros}, denoted by $\zer M:=\{x \in \R^n: M(x) = 0\}$, is nonempty. The operator $M:\R^n \rightarrow \R^n$ is said to be monotone if $\langle M(x) - M(y), x-y \rangle \geq 0$ for all $x,y \in \R^n$.

For $M := \nabla f$, where $f : \R^n \rightarrow \R$ is a convex and differentiable function with a Lipschitz continuous gradient, solving equation \eqref{eq:find-zeros} is equivalent to finding the global solutions of the optimization problem
\begin{align}\label{opt}
    \min_{x\in\R^n} f(x).
\end{align}

On the other hand, if we define
\begin{align*}
 M(x,y) := (\nabla_x \Phi(x,y), -\nabla_y \Phi(x,y)),
\end{align*}
where $\Phi : \R^p \times \R^q \rightarrow \R$ is a convex-concave function with Lipschitz continuous gradient, solving equation \eqref{eq:find-zeros} is equivalent to finding the saddle points of the minimax problem
\begin{align}\label{minimax}
    \min_{x\in\R^p}\max_{y\in\R^q} \Phi(x,y).
\end{align}

Inspired by the Newton and the Levenberg-Marquardt methods, Attouch and Svaiter introduced in ~\cite{attouch-svaiter} the following continuous time equation associated to \eqref{eq:find-zeros} in the setting of a real Hilbert space
\begin{align}\label{eq:deterministic-system}
    \begin{cases}
        &\dot{x}(t) + \mu(t) \frac{d}{dt}M(x(t)) + \mu(t) M(x(t)) = 0 \quad \forall t > 0,
        \\& x(0) = x_0,
    \end{cases}
\end{align}
where $\mu : [0,+\infty) \rightarrow (0,+\infty)$ was assumed to be bounded, continuous and absolutely continuous on bounded intervals. Besides proving the existence and uniqueness of strong global solutions, they showed, under suitable assumptions on the parameter function, that  $x(t)$ converges weakly to a zero of $M$ as well as $M(x(t)) \rightarrow 0$ as $t\rightarrow +\infty$. 

In many instances, the evaluation of the operator is subject to noise, which  motivates us to transfer ~\eqref{eq:deterministic-system} to a stochastic setting. To achieve this, we consider the formal expression of a continuous time system
\begin{align}\label{eq:SDE-formal}
\begin{cases}
    &dX(t) + \mu(t) dM(X(t)) = -\gamma(t)M(X(t)) dt + \sigma(t,X(t)) dW(t) \quad \forall t > 0,
    \\& X(0)=X_0,
\end{cases}
\end{align}
which can consistently with It\^{o}'s chain rule be understood as
\begin{align}
\tag{SDE-M}
\label{eq:SDE-M}
\begin{split}
    \begin{cases} &d(X(t)+\mu(t)M(X(t))) = - (\gamma(t) - \dot{\mu}(t)) M(X(t)) dt + \sigma(t,X(t)) dW(t) \quad \forall t > 0,
    \\& X(0)=X_0,
    \end{cases}
\end{split}
\end{align}
defined over a filtered probability space $(\Omega,\mathcal{F},\{\mathcal{F}_t\}_{t\geq 0},\mathbb{P})$ with the diffusion term $\sigma: \R_+ \times \R^n \rightarrow \R^{n \times m}$ being matrix-valued and measurable, $W$ a $m$-dimensional Brownian motion, and $X(\cdot)$ as well as $M(X(\cdot))$ are stochastic It\^{o} processes with the same $m$-dimensional Brownian motion $W$. The parameter functions $\mu : [0,+\infty) \rightarrow (0,+\infty)$ and $\gamma : [0,+\infty) \rightarrow (0,+\infty)$ are assumed to be continuous differentiable and, respectively, integrable. For the diffusion term we assume 
\begin{align}\label{eq:assumption-sigma}
\begin{cases}
    & \exists c_\sigma >0 \ \mbox{such that} \ \|\sigma(t,x')-\sigma(t,x)\|_F \leq c_\sigma\|x'-x\| \quad \forall t >0,\\
    & \sigma_\infty(t) := \sup_{x\in\R^n} \|\sigma(t,x)\|_F \leq \sigma_* \quad \forall t >0,
\end{cases}
\end{align}
where $\|\cdot\|_F$ denotes the Frobenius norm on $\R^{n \times m}$. 

Then,~\eqref{eq:SDE-M} can be rewritten as the system
\begin{align*}
\begin{cases}
    &dM(X(t)) = Y(t) dt + \sigma_M(t) d{W}(t),
    \\& dX(t) = (-\mu(t)Y(t) - \gamma(t)M(X(t))) dt + \sigma_X(t) d{W}(t) \quad \forall t > 0,
    \\& X(0)=X_0,
\end{cases}
\end{align*}
where
$$\sigma(t,X(t)) = \sigma_X(t) + \mu(t) \sigma_M(t)\quad \forall t \geq 0$$
and $\sigma_X, \sigma_M: \R_+ \rightarrow \R^{n \times m}$ are measurable.

Our goal is to explore the existence and uniqueness of trajectory solutions for \eqref{eq:SDE-M}, along with their long-term behavior in terms of both almost sure convergence and convergence rates. Furthermore, we demonstrate how these properties transfer to convergence rates for stochastic numerical methods obtained through temporal discretizations of the stochastic differential equation.

\subsection{Related works}\label{subsec11}

The simplest continuous time approach associated to the optimization problem \eqref{opt}, for a convex and differentiable function $f: \R^n \rightarrow \R$  with $L_{\nabla f}$-Lipschitz continuous gradient, is the gradient flow
\begin{align*}
\begin{cases}
    &\dot{x}(t) = -\nabla f(x(t)) \quad \forall t>0,
    \\& x(0) = x_0.
\end{cases}
\end{align*}
Its trajectory solution $x(t)$ converges to a global minimizer of $f$, while $f(x(t))$ converges to the minimal function value $\min f$ of $f$ with a convergence rate of $o(\frac{1}{t})$ as $t \rightarrow +\infty$.

Its discrete time counterpart, the gradient method, 
\begin{align*}
    x^{k+1} = x^k - \gamma \nabla f(x^k) \quad \forall k \geq 0,
\end{align*}
where $x^0 \in \R^n$ and $\gamma \in \left(0, \frac{1}{L_{\nabla f}}\right ]$, shares the same convergence properties. Specifically, the sequence $(x_k)_{k \geq 0}$ converges to a minimizer of $f$, and $f(x_k) - \min f = o(\frac{1}{k})$ as $k \rightarrow +\infty$.

In \cite{s.-fadili-attouch}, to accommodate instances where the gradient input is subject to noise, Maulen-Soto, Fadili and Attouch proposed the following stochastic differential equation
\begin{align}\label{sdeopt}
\begin{cases}
    &dX(t) = -\nabla f(X(t)) dt + \sigma(t,X(t)) dW(t) \quad \forall t>0,
    \\& X(0) = X_0,
\end{cases}
\end{align}
defined over a filtered probability space $(\Omega,\mathcal{F},\{\mathcal{F}_t\}_{t\geq 0},\mathbb{P})$ with $\sigma: \R_+ \times \R^n \rightarrow \R^{n \times m}$ being a matrix-valued measurable diffusion term fulfilling \eqref{eq:assumption-sigma}, and $W$ a $m$-dimensional Brownian motion. They derived for $f\left(\int_0^t X(s)ds \right) - \min f$ and $\frac{1}{t}\int_0^t f(X(s))ds - \min f$ upper bounds in expectation of ${\cal O}(\frac{1}{t}) + {\cal O}(\sigma^2_*)$.  Additionally, assuming that $\sigma_\infty$ is square-integrable, they demonstrated the almost sure convergence of the trajectory process to a minimizer of $f$ and of $f(X(t))$ to $\min f$ as $t \rightarrow +\infty$, and that the two above quantities exhibit convergence rates of ${\cal O}(\frac{1}{t})$ as $t \rightarrow +\infty$.

There is a strong link between this stochastic differential equation and the stochastic gradient descent algorithm with constant stepsize (\cite{chaudhari-sotto},~\cite{hu-li-li-liu},~\cite{li-tai-e},~\cite{li-malladi-arora},~\cite{orvieto-lucchi},~\cite{shi-su-jordan})
\begin{align*}
    x^{k+1} := x^k - \gamma \nabla f(x^k, \xi^k) \quad \forall k \geq 0,
\end{align*}
where $({\cal F}_k)_{k \geq 0}$ is an increasing family of $\sigma$-fields, $x^0 \in \R^n$ is ${\cal F}_0$-measurable,  $(\xi_k)_{k \geq 0}$ is such that $\nabla f(x, \xi^k)$ corresponds to a noisy observation of $\nabla f(x)$, with the properties that $\mathbb{E}(\nabla f(x, \xi^{k+1})|{\cal F}_k)= \nabla f(x)$ for all $x \in \R^n$ and all $k \geq 0$, and $\mathbb{E}(\|\nabla f(x^*, \xi^{k+1})\|^2|{\cal F}_k) \leq \sigma_{**}^2$ for a global minimizer $x^*$ of $f$ and all $k \geq 0$. Then (see, for instance, \cite{bach-moulines}) the upper bound in expectation of $f\left(\frac{1}{k}\sum_{i=0}^{k-1} x^i \right) - \min f$ is of ${\cal O}(\frac{1}{k}) + {\cal O}(\sigma^2_{**})$.

Recently, in \cite{mfa-tikh}, \eqref{sdeopt} has been enhanced with a Tikhonov regularization term, resulting in a stochastic differential equation which, under the same upper bounds and convergence rates in expectation, in the case when $\sigma_\infty$ is square-integrable, exhibits strong convergence of the trajectory process to the global minimizer of $f$ of minimum norm.

Furthermore, in \cite{mfao-secondorder}, the properties of first order stochastic differential equations have been transferred to second order ones using the time scaling and averaging methodology developed in \cite{attouch-bot-nguyen}. The latter generalize the second order Langevin process, and and demonstrate rapid convergence in expectation of the function values along the trajectory process to the minimal function value.

As for the monotone equation \eqref{eq:find-zeros}, one could replicate the gradient flow approach and associate with it the following continuous dynamics
\begin{align}\label{dynamicsmon}
\begin{cases}
    &\dot{x}(t) = -M(x(t)) \quad \forall t>0,
    \\& x(0) = x_0.
\end{cases}
\end{align}
However, only the ergodic trajectory $\frac{1}{t} \int_0^t x(s)ds$ converges to a zero of $M$ as $t \rightarrow +\infty$ (\cite{baillon-brezis}), while the trajectory $x(t)$ does not converge in general, as can be easily seen for the case of the counterclockwise rotation operator by $\frac{\pi}{2}$ radians in $\R^2$ (\cite{cominetti-peypouquet-sorin}). 

If the operator $M$ is $\rho$-cocoercive with constant $\rho >0$, meaning that
$$\langle M(x) - M(y), x-y \rangle \geq \rho \|M(x) - M(y)\|^2 \quad \forall x,y \in \R^n,$$
which obviously implies that $M$ is monotone and $\frac{1}{\rho}$-Lipschitz continuous, then the asymptotic behaviour of the trajectory of \eqref{dynamicsmon} improves substantially, even in the presence of small perturbations. Precisely, according to \cite[Theorem 11]{attouch-bot-nguyen}, in this case, the trajectory solution of
\begin{align}\label{dynamicsmonpert}
\begin{cases}
    &\dot{x}(t) = -M(x(t)) + g(t) \quad \forall t>0,
    \\& x(0) = x_0,
\end{cases}
\end{align}
where $g : [0, +\infty) \rightarrow \R^n$ is such that
\begin{equation}\label{cond:g}
\int_{t_0}^{+\infty} \|g(t)\| dt < +\infty \ \mbox{and} \ \int_{t_0}^{+\infty} t\|g(t)\|^2 dt < +\infty,
\end{equation}
converges to a zero of $M$, and it holds $\|M(x(t))\| = o(\frac{1}{\sqrt{t}})$ as $t \rightarrow + \infty$.

In the case when $M$ is a $\rho$-cocoercive operator with constant $\rho >0$, the following stochastic counterpart to \eqref{dynamicsmon} has been proposed in \cite{s.-fadili-attouch}
\begin{align}\label{sdemon}
\begin{cases}
    &dX(t) = -M(X(t)) dt + \sigma(t,X(t)) dW(t) \quad \forall t>0,
    \\& X(0) = X_0,
\end{cases}
\end{align}
defined over a filtered probability space $(\Omega,\mathcal{F},\{\mathcal{F}_t\}_{t\geq 0},\mathbb{P})$, where $\sigma: \R_+ \times \R^n \rightarrow \R^{n \times m}$ fulfills \eqref{eq:assumption-sigma}, and $W$ a $m$-dimensional Brownian motion. Alongside the existence and uniqueness of the trajectory process for \eqref{sdemon}, upper bounds in expectation of ${\cal O}(\frac{1}{t}) + {\cal O}(\sigma^2_*)$ for $\|\frac{1}{t} \int_0^t M(X(s))ds\|^2$ and $\frac{1}{t} \int_0^t \|X(s)\|^2ds$ have been derived. In case $M$ is $\kappa$-strongly monotone with constant $\kappa >0$, upper bounds in expectation of ${\cal O}(e^{-2\kappa t}) + {\cal O}(\sigma^2_*)$ for $\|X(t)-x^*\|^2$, where $x^*$ is the unique zero of $M$, have been also provided. Additionally, assuming that $\sigma_\infty$ is square-integrable,  it has been proved that there is almost sure convergence of $X(t)$ towards a zero of $M$ and of $\|M(X(t))\|$ towards $0$ as $t \rightarrow +\infty$. Moreover, $\|\frac{1}{t} \int_0^t M(X(s))ds\|^2$ and $\frac{1}{t} \int_0^t \|X(s)\|^2ds$ exhibit convergence rates of ${\cal O}(\frac{1}{t})$ as $t \rightarrow +\infty$.

If $M$ is only monotone and Lipschitz continuous, \cite{attouch-svaiter} suggests that incorporating a correction term into the formulation of the differential equation, such as the time derivative of $t \mapsto M(x(t))$ in \eqref{eq:deterministic-system}, is crucial to improve the convergence properties of the dynamical system associated with \eqref{eq:find-zeros}. We have extended this idea to the stochastic domain, resulting in the formulation of \eqref{eq:SDE-M}, and assuming that $M(X(\cdot))$ is an stochastic It\^o process.

The significance of the correction term in solving equations governed by monotone and $L$-Lipschitz continuous operators is well-reflected in the realm of numerical algorithms. On one hand, the forward method,
\begin{align*}
    x^{k+1} := x^k - \gamma M(x^k) \quad \forall k \geq 0,
\end{align*}
does not generally converge, as observed, for instance, with the counterclockwise rotation operator by $\frac{\pi}{2}$ radians in $\mathbb{R}^2$.

On the other hand, for the Optimistic Gradient Descent Ascent method (\cite{popov}, \cite{boehm-sedlmayer-csetnek-bot})
\begin{align*}
    x^{k+1} := x^k - 2\gamma M(x^k) + \gamma M(x^{k-1}) = x^k - \gamma M(x^k) - \gamma(M(x^k) - M(x^{k-1})) \quad \forall k \geq 1,
\end{align*}
which can be regarded as a natural discretization of \eqref{eq:deterministic-system}, if $0 < \gamma < \frac{1}{2L}$, then the generated sequence of iterates converges to a zero of $M$. Additionally, if $0 < \gamma < \frac{1}{16L}$, then the method exhibits a best-iterate convergence rate for the norm of the operator evaluated along the iterates of $\mathcal{O}\left(\frac{1}{\sqrt{k}}\right)$ as $k \rightarrow +\infty$ (\cite{chavdarova-jordan-zampetakis}, \cite{golowich-pattathil-daskalakis}, \cite{golowich-pattathil-daskalakis-ozdaglar}).

Furthermore, for the Extragradient method (\cite{antipin},~\cite{korpelevich})
\begin{align*}
\begin{split}
    \begin{cases}
        & y^k := x^k - \gamma M(x^k) \\
        & x^{k+1} := x^k - \gamma M(y^k)
    \end{cases}
\end{split}     \quad \forall k \geq 0,
\end{align*}
which also involves a correction term, if $0<\gamma<\frac{1}{L}$, then the generated sequence of iterates converges to a zero of $M$ and $\|M(x^k)\|$ converges to $0$ with a convergence rate of $\mathcal{O}\left(\frac{1}{\sqrt{k}}\right)$ as $k \rightarrow +\infty$ (\cite{gorbunov-loizou-gidel}).

For completeness, we would like to mention that recently, in \cite{bot-csetnek-nguyen}, a second-order dynamical system with momentum term and correction term associated with a monotone equation has been introduced. Alongside the convergence of the generated trajectory to a zero of the operator, the system exhibits convergence rates of $o\left(\frac{1}{t}\right)$ for the norm of the operator and the gap function along the trajectory as $t \rightarrow +\infty$. Furthermore, an explicit numerical algorithm replicating the convergence properties of the continuous time dynamics has been obtained. These are the best-known convergence rate results for continuous time and explicit discrete time approaches for monotone equations.

\subsection{Our contributions}\label{subsec12}

In Section~\ref{sec-cont-time-ex}, we discuss the existence and uniqueness of a trajectory solution for the continuous time system ~\eqref{eq:SDE-M}. The presence of the corrector term requires quite involved analysis. Under the assumption that $\sigma_\infty$ is square-integrable, we also prove the almost sure convergence of $X(t)$ to a zero of $M$ and of $\|M(X(t))\|$ to $0$ as $t \rightarrow +\infty$. 

In Section~\ref{sec-cont-time-rates}, in the framework of assumption \eqref{eq:assumption-sigma} for the diffusion term, we establish upper bounds in expectation for $\frac{1}{t} \int_0^t\|M(X(s))\|^2ds$ and $\frac{1}{t} \int_0^t \langle X(s)-x^*, M(X(s))\rangle ds$, where $x^*$ is a zero of $M$. These bounds are of the form ${\cal O}(\frac{1}{t\mu(t)})$ + ${\cal O}(\sigma^2_{*})$ and ${\cal O}(\frac{1}{t})$ + ${\cal O}(\sigma^2_{*})$, respectively. Additionally, assuming that $\sigma_\infty$ is square-integrable, we show that the squared norm of the operator $\|M(X(t))\|^2$ and the gap function $\langle X(t)-x^*, M(X(t))\rangle$ exhibit ergodic convergence rates in expectation of ${\cal O}(\frac{1}{t\mu(t)})$ and ${\cal O}(\frac{1}{t})$, respectively, as $t \rightarrow + \infty$.  In case $M$ is $\kappa$-strongly monotone with constant $\kappa >0$, we also derive upper bounds in expectation of ${\cal O}(e^{-2\kappa t}) + {\cal O}(\sigma^2_*)$ for $\|X(t)-x^*\|^2$, where $x^*$ is the unique zero of $M$. 

In Section~\ref{sec-disc-time}, we demonstrate how these results transfer from continuous time to discrete time settings within the context of the stochastic Optimistic Gradient Descent Ascent (OGDA) method and of the stochastic Extragradient (EG) method, designed to solve  \eqref{eq:find-zeros}, in the circumstances that the evaluation of $M$ is subject to noise.

\section{Continuous time system: existence, uniqueness and almost sure convergence}\label{sec-cont-time-ex}

This section is dedicated to the study of~\eqref{eq:SDE-M} in the setting 
\begin{align*}
\begin{split}
    \begin{cases} &d(X(t)+\mu(t)M(X(t))) = - (\gamma(t) - \dot{\mu}(t)) M(X(t)) dt + \sigma(t,X(t)) dW(t) \quad \forall t > 0,
    \\& X(0)=X_0.
    \end{cases}
\end{split}
\end{align*}
Throughout this section, we will assume that the parameter functions satisfy the following conditions:
\begin{align}\label{condmugamma}
 0 < \mu_\text{low} \leq \mu(t) \leq \mu_\text{up}:=\mu(0) \ \mbox{and} \ 0 < \gamma_\text{low} \leq \gamma(t) \leq \gamma_\text{up} \quad \forall t\geq 0, \ \mbox{and} \ \mu \ \mbox{is nonincreasing on} \ [0,+\infty).
\end{align}

Throughout the paper, $(\Omega,\mathcal{F},\{\mathcal{F}_t\}_{t\geq 0},\mathbb{P})$ is a filtered probability space. The expectation of a random variable $\xi : \Omega \rightarrow \R^n$ is denoted by $\mathbb{E}(\xi) :=\int_{\Omega} \xi(\omega) d\mathbb{P}(\omega)$. An event $E \in {\cal F}$ happens almost surely (a.s.) if $\mathbb{P}(E)=1$.

An $\R^n$-valued stochastic process is a function $X:\Omega \times \R_+ \rightarrow \R^n$. It is said to be continuous if $X(\omega, \cdot)$ is continuous on $\R_+$ for almost all $\omega \in \Omega$. For simplicity, we will denote $X(t):=X(\cdot,t)$. Two stochastic processes $X, Y:[0,T]\rightarrow \R^n$, for $T>0$, are said to be equivalent if $X(t)=Y(t)$ almost surely for all $ t\in[0,T]$. This allows the definition of the equivalence relation $\mathcal{R}$ which associates equivalent stochastic processes to the same class.

Next, we will introduce some notions that will characterize the space where the trajectory solution of \eqref{eq:SDE-M} lies.

\begin{defi}\label{def21}
(i) A stochastic process $X:\Omega\times\R_+ \rightarrow \R^n$ is called progressively measurable if for every $t\geq 0$, the mapping
        \begin{align*}
            &\Omega\times [0,t] \rightarrow \R^n, \quad (\omega,s) \mapsto X(\omega,s)
        \end{align*}
        is $\mathcal{F}_t\otimes\mathcal{B}([0,t])$-measurable, where $\otimes$ denotes the product $\sigma$-algebra and $\mathcal{B}$ is the Borel $\sigma$-algebra. Further, $X$ is called $\mathcal{F}_t$-adapted if $X(\cdot,t)$ is $\mathcal{F}_t$-measurable for every $t\geq 0$.
        \item For $T>0$, we define the quotient space as
        \begin{align*}
            S_n^0[0,T]:=\left\{X:\Omega\times[0,T]\rightarrow\R^n: \text{$X$ is a progressively measurable continuous stochastic process}\right\}\Big/\mathcal{R},
        \end{align*}
     and set $S_n^0:=\bigcap_{T\geq 0} S_n^0[0,T]$. 
     
(ii) For $\nu>0$ and $T >0$, we define $S_n^\nu [0,T]$ as the following subset of stochastic processes in $S_n^0 [0,T]$
        \begin{align*}
            S_n^\nu [0,T] := \left\{X\in S_n^0 [0,T]: \quad \mathbb{E}\left(\sup_{t\in[0,T]} \|X(t)\|^\nu\right) < +\infty \right\}.
        \end{align*}
        Finally, we set $S_n^\nu := \bigcap_{T\geq 0} S_n^\nu[0,T]$.

\end{defi}

The following results establishes the existence and uniqueness of a solution of~\eqref{eq:SDE-M}.

\begin{thm}\label{thm:existence-uniqueness}
    Let the diffusion term $\sigma : \R_+ \times \R^n \rightarrow \R^{n \times m}$ satisfy assumption ~\eqref{eq:assumption-sigma}. Then ~\eqref{eq:SDE-M} has a unique solution $X\in S_n^\nu$, where $\nu \geq 2$.
\end{thm}

\begin{proof} Let $\nu \geq 2$. First, we establish the result on $[0,T]$, for $T>0$ fixed.  We rewrite the stochastic differential equations as
    \begin{align*}
        X(t) = (\Id + \mu(t)M)^{-1}(\underbrace{X(t) + \mu(t)M(X(t))}_{=:Z(t)}) =: J_{\mu(t)M}(Z(t)) \quad \forall t\in [0,T],
    \end{align*}
    where $J_{\mu(t)M} := (\Id + \mu(t)M)^{-1} : \R^n \rightarrow \R^n$ denotes the resolvent of $M$ with parameter $\mu(t)$ and $\Id$ denotes the identity operator on $\R^n$. Denoting the Yosida approximation of $M$ with parameter $\mu(t)$ by
    \begin{align*}
       M_{\mu(t)}:= \frac{1}{\mu(t)}(\Id-J_{\mu(t)M}),
    \end{align*}   
    a simple calculation yields that $M(X(t)) = M_{\mu(t)}(Z(t))$ for all $t \in [0,T]$. Hence,
    \begin{align}\label{eq:existence-uniqueness-minty}
        dZ(t) = d(X(t) + \mu(t)M(X(t))) = (-\gamma(t) + \dot{\mu}(t)) M_{\mu(t)}(Z(t)) dt + \sigma(t,J_{\mu(t)M}(Z(t))) dW(t) \quad \forall t\in [0,T].
    \end{align}
    Now setting
    \begin{align*}
        F(t,z) := (\dot{\mu}(t) - \gamma(t)) M_{\mu(t)}(z),\quad G(t,z) := \sigma(t,J_{\mu(t)M}(z)) \quad \forall z \in \R^n \ \forall t \in [0,T],
    \end{align*}
    yields for all $x,y \in \R^n$ and all $t \in [0,T]$
    \begin{align*}
        & \|F(t,x)-F(t,y)\| + \|G(t,x) - G(t,y)\|_F \\
        = & \ |\gamma(t)-\dot{\mu}(t)|\|M_{\mu(t)}(x) - M_{\mu(t)}(y)\| + \|\sigma(t,J_{\mu(t)M}(x)) - \sigma(t,J_{\mu(t)M}(y))\|_F\\
        \leq & \ (\gamma_\text{up} - \dot{\mu}_{\text{low},T}) \frac{1}{\mu(t)}\|x-y\| + c_\sigma\|J_{\mu(t)M}(x) - J_{\mu(t)M}(y)\|\\
        \leq & \ \left((\gamma_\text{up} - \dot{\mu}_{\text{low},T})\frac{1}{\mu_\text{low}}+c_\sigma\right)\|x-y\|,
    \end{align*}
where $\dot{\mu}_{\text{low},T}:=\inf\{\dot{\mu}(t) : t \in [0,T]\} \in \R$, since the  continuous mapping $\dot{\mu}$ is bounded on the compact interval $[0,T]$. 

This allows us to use Theorem~\ref{thm:help-existence-uniqueness} to obtain the existence of a unique stochastic process $Z_T\in S_n^\mu[0,T]$ solving~\eqref{eq:existence-uniqueness-minty}. Therefore, a reasonable candidate to solve~\eqref{eq:SDE-M}, given by
    \begin{align*}
        X_T(t) := J_{\mu(t)M}(Z_T(t)) \quad \forall t \in [0,T],
    \end{align*}
    is uniquely defined. It remains to show that it does indeed solve~\eqref{eq:SDE-M} and lies in $S_n^\nu[0,T]$. 
    
We define
    \begin{align*}
        Z_T(t) := X_T(t) + \mu(t)M(X_T(t)) \quad \forall t \in [0,T].
    \end{align*}
    Since now $M(X_T(t)) = M_{\mu(t)}(Z_T(t))$ and $Z_T$ is a solution of~\eqref{eq:existence-uniqueness-minty}, it follows
    \begin{align*}
        d(X_T(t)+\mu(t)M(X_T(t))) &= (-\gamma(t) + \dot{\mu}(t))M_{\mu(t)}(Z_T(t)) dt + \sigma(t,J_{\mu(t)M}(Z_T(t))) dW(t)
        \\&= (-\gamma(t) + \dot{\mu}(t))M(X_T(t)) dt + \sigma(t,X_T(t)) dW(t) \quad \forall t\in[0,T],
    \end{align*}
    and thus, $X_T$ is indeed a solution of~\eqref{eq:SDE-M} on $[0,T]$.

    Now, recall that $J_{\mu(t)M}$ is $1$-Lipschitz continuous and observe that for all $x\in\R^n$ and all $\mu,\lambda>0$ it holds
    \begin{align*}
        \|J_{\lambda M}(x) - J_{\mu M}(x)\| &= \left\| J_{\mu M}\left(\frac{\mu}{\lambda}x+\left(1-\frac{\mu}{\lambda}\right)J_{\lambda M}(x)\right)-J_{\mu M}(x)\right\|
        \\&\leq \left\|\left(1-\frac{\mu}{\lambda}\right)(x-J_{\lambda M}(x))\right\|
        \\&\leq |\lambda - \mu| \|M_\lambda x\|.
    \end{align*}
This will allow us to establish the continuity of $X_T$. Indeed, for all $t,s \in [0,T]$ it holds
    \begin{align*}
        \|X_T(t)-X_T(s)\| &\leq \|J_{\mu(t)M}(Z_T(t)) - J_{\mu(s)M}(Z_T(t))\| + \|J_{\mu(s)M}(Z_T(t)) - J_{\mu(s)M}(Z_T(s))\|
        \\&\leq |\mu(t)-\mu(s)| \|M_{\mu(t)}(Z_T(t))\| + \|Z_T(t)-Z_T(s)\|,
    \end{align*}
    and, since $Z_T$ is continuous by Theorem~\ref{thm:help-existence-uniqueness} and $\mu$ is continuous by assumption, the continuity of $X_T$ follows.

Next, consider
    \begin{align*}
        \mathbb{E}\left(\sup_{t\in[0,T]} \|X_T(t)\|^\nu\right) &= \mathbb{E}\left(\sup_{t\in[0,T]} \|J_{\mu(t)M}(Z_T(t)) - 0\|^\nu\right)
        \\&= \mathbb{E}\left(\sup_{t\in[0,T]} \|J_{\mu(t)M}(Z_T(t)) - J_{\mu(t)M}(\mu(t)M(0))\|^\nu \right)
        \\&\leq \mathbb{E}\left(\sup_{t\in[0,T]} \|Z_T(t) - \mu(t)M(0)\|^\nu\right)
        \\&\leq \mathbb{E}\left(2^\nu \sup_{t\in[0,T]} \left(\frac{1}{2}\|Z_T(t)\| + \frac{\mu_\text{up}}{2}\|M(0)\|\right)^\nu\right)
        \\&\leq 2^{\nu-1}\left(\mathbb{E}\left(\sup_{t\in[0,T]}\|Z_T(t)\|^\nu\right) + \mu_\text{up}^\nu \|M(0)\|^\nu\right)<+\infty.
    \end{align*}

Finally, in order to prove the progressive measurability of $X$, consider the mapping
    \begin{align*}
        \Omega \times [0,t] \rightarrow [0,t]\times \R^n, \quad (\omega,s) \mapsto (s,Z_T(\omega,s)).
    \end{align*}
    Now, the preimage of an element of the generating set $I\times O$, where $I\subseteq [0,t]$ is an interval and $O\subseteq \R^n$ an element of the generator of the Borel-$\sigma$-algebra on $\R^n$, is given by
    \begin{align*}
        Z_T^{-1}(I\times O) \cap I\times \Omega,
    \end{align*}
    which is measurable, due to the progressive measurability of $Z_T$ and the fact that a finite intersection of measurable sets is again measurable. This proves the progressive measurability of $X_T$.

    Since $X_T$ is uniquely defined on any $[0,T]$, where $T>0$, for $0<T_1<T_2$, we have $X_{T_2}\vert_{[0,T_1]} = X_{T_1}$. Thus, there exists a unique solution $X\in S_n^\nu$ of~\eqref{eq:SDE-M}.
\end{proof}

The following result establishes the almost sure convergence of $X(t)$ to a $\zer M$-valued random variable as $t \rightarrow +\infty$. The notion of a martingale will play an important role in its proof.

\begin{defi}\label{def22}
    Let $X$ be a real-valued stochastic process such that $\mathbb{E}(|X(t)|)<+\infty$ for all $t\geq 0$.
    \begin{enumerate}[(i)]
        \item The $\sigma$-algebra generated by the random variables $X(s)$ for $0\leq s \leq t$,
        \begin{align*}
            \mathcal{U}(t) := \mathcal{U}(X(s)\vert 0\leq s \leq t),
        \end{align*}
        is called the \emph{history} of the stochastic process $X$ until (and including) time $t$.
        \item If
        \begin{align*}
            X(s) = \mathbb{E}(X(t) \vert \mathcal{U}(s)) \quad \text{a.s.} \quad \forall t \geq s \geq 0,
        \end{align*}
        then $X$ is called a \emph{martingale}.
    \end{enumerate}
\end{defi}
\begin{thm}\label{thm:convergence-tdp}
Let $\zer M \neq \emptyset$, the diffusion term $\sigma : \R_+ \times \R^n \rightarrow \R^{n \times m}$ satisfy assumption ~\eqref{eq:assumption-sigma}, and $X$ be the unique trajectory process of ~\eqref{eq:SDE-M}. If, in addition to \eqref{condmugamma}, $\int_0^{+\infty}\sigma_\infty(s)^2 ds < +\infty$, then the following statements are true:
    \begin{enumerate}
        \item[(i)] $\sup_{t \geq 0}{\mathbb{E}(\|X(t)\|^2)}<+\infty$;
        \item[(ii)]  it holds $\sup_{t\geq 0}{\|X(t)\|} < +\infty$ a.s.
            \end{enumerate}
If, in addition, $-\infty < \dot \mu_\text{up} \leq \dot \mu(t)$ for all $t \geq 0$ and $L \mu_\text{up}<1$, then:          
         \begin{enumerate}
        \item[(iii)]\label{en-item:limit-exists-tdp} almost surely $\lim_{t\rightarrow +\infty}{\|M(X(t))\|}=0$ and $\lim_{t\rightarrow +\infty} \|X(t)-x^*\|$ exists and is finite for all $x^* \in \zer M$;
        \item[(iv)] there exists a $\zer M$-valued random variable $x^*$ such that $\lim_{t\rightarrow +\infty}X(t)=x^*$ a.s.
    \end{enumerate}
\end{thm}
\begin{proof}
In order to perform the proof, we rewrite \eqref{eq:SDE-M} as
\begin{align*}
\begin{cases}
    &dM(X(t)) = Y(t) dt + \sigma_M(t) d{W}(t),
    \\& dX(t) = (-\mu(t)Y(t) - \gamma(t)M(X(t))) dt + \sigma_X(t) d{W}(t) \quad \forall t > 0,
    \\& X(0)=X_0,
\end{cases}
\end{align*}
where $\sigma(t,X(t)) = \sigma_X(t) + \mu(t)\sigma_M(t)$ for all $t \geq 0$. Let $x^*\in\zer M$, and define the anchor function $\phi : \R_+ \times \R^n \times \R^n \rightarrow \R$,
    \begin{align*}
        \phi(t,x,z) := \mu(t) \langle z, x-x^*\rangle+ \frac{1}{2}\|x-x^*\|^2 + \frac{\mu(t)^2}{2}\|z\|^2 = \frac{1}{2}\|x+\mu(t)z-x^*\|^2.
    \end{align*}
The proposed anchor function takes an additional argument -- this will be taken equal to $M(X(t))$ over the course of the proof. For all $(t,x,z) \in \R_+ \times \R^n \times \R^n$ it holds
    \begin{align*}
        &\frac{d}{dt}\phi(t,x,z) = \dot{\mu}(t)\langle z, x-x^*\rangle + \mu(t)\dot{\mu}(t)\|z\|^2 
        \\& \nabla_x \phi(t,x,z) = \mu(t)z + x - x^*
        \\& \nabla_z \phi(t,x,z) = \mu(t)(x-x^*) + \mu(t)^2 z
        \\& \nabla_x^2 \phi(t,x,z) = I
        \\& \nabla_x\nabla_z \phi(t,x,z) = \nabla_z\nabla_x \phi(t,x,z) = \mu(t)I
        \\& \nabla_z^2 \phi(t,x,z) = \mu(t)^2 I.
    \end{align*}
Observe that the remaining second partial derivatives of $\phi$ are identically $0$. Hence, the It\^{o} formula yields for all $t\geq 0$
    \begin{align*}
        \phi(t,X(t),M(X(t))) &= \phi(0,X(0),M(X(0)))
        \\&\qquad+ \int_0^t \left(\dot{\mu}(s)\langle M(X(s)), X(s)-x^*\rangle + \mu(s)\dot{\mu}(s)\|M(X(s))\|^2\right) ds
        \\&\qquad+ \int_0^t \big\langle (\mu(s)M(X(s)) + X(s) - x^*), \left( -\mu(s)Y(s) - \gamma(s) M(X(s)) \right) \big\rangle ds
        \\&\qquad+ \int_0^t \left\langle \sigma_X^\top(s)(\mu(s)M(X(s)) + X(s) - x^*), dW(s) \right\rangle
        \\&\qquad+ \int_0^t \left\langle (\mu(s)(X(s)-x^*)+\mu(s)^2M(X(s))),Y(s) \right\rangle ds
        \\&\qquad+ \int_0^t \left\langle \sigma_M^\top(s) (\mu(s)(X(s)-x^*) + \mu(s)^2 M(X(s))), dW(s) \right\rangle
        \\&\qquad+ \frac{1}{2}\int_0^t \text{tr}(\sigma_X^\top(s) \sigma_X(s)) + 2\mu(s)\text{tr}(\sigma_M^\top(s) \sigma_X(s)) + \mu(s)^2\text{tr}(\sigma_M^\top(s) \sigma_M(s)) ds.
    \end{align*}
We observe for all $s\geq 0$
    \begin{align*}
    \sigma_X^\top(s) \sigma_X(s) + 2\mu(s)\sigma_M^\top(s) \sigma_X(s) + \mu(s)^2(\sigma_M^\top(s) \sigma_M(s)) = \sigma^\top(s,X(s))\sigma(s,X(s)),
    \end{align*}
    and denote this term by $\Sigma(s,X(s))$ for improved readability. Further, it holds for all $s \geq 0$
    \begin{align*}
       & \ \langle \sigma_X^\top(s)(\mu(s)M(X(s)) + X(s)-x^*) + \sigma_M^\top(s)(\mu(s)^2 M(X(s)) + \mu(s)(X(s)-x^*)), d{W}(s)\rangle
        \\
        = & \ \langle (\sigma_X(s) + \mu(s)\sigma_M(s))^\top(\mu(s)M(X(s)) + X(s)-x^*), d{W}(s)\rangle
        \\
        = & \ \langle \sigma^\top(s,X(s))(\mu(s)M(X(s)) + X(s)-x^*), d{W}(s)\rangle.
    \end{align*}
    Then, performing some elementary computations, we get for all $t\geq 0$
    \begin{align}\label{eq:ito-formula-anchor-function}
    \begin{split}
        \phi(t,X(t),M(X(t))) &= \underbrace{\mu(0)\langle M(X_0), X_0-x^*\rangle + \frac{1}{2}\|X_0-x^*\|^2 + \frac{\mu(0)^2}{2}\|M(X_0)\|^2}_{=:\xi}
        \\&\quad - \underbrace{\int_0^t (\gamma(s) - \dot{\mu}(s))\left(\langle M(X(s)), X(s)-x^*\rangle +\mu(s)\|M(X(s))\|^2\right)  ds}_{=:U(t)}
        \\&\quad +\underbrace{\int_0^t \langle \sigma^\top(s,X(s))(\mu(s)M(X(s)) + X(s)-x^*), d{W}(s)\rangle}_{=:N(t)}
        \\&\quad+ \underbrace{\frac{1}{2}\int_0^t \text{tr}(\Sigma(s,X(s))) ds}_{=:A(t)},
    \end{split}
    \end{align}
    in other words
    \begin{align*}
        \phi(t,X(t),M(X(t))) = \xi - U(t) + N(t) + A(t) \quad \forall t\geq 0.
    \end{align*}

(i) As $X\in S_n^2$, it holds for all $T>0$ that
        \begin{align*}
            & \ \mathbb{E}\left(\int_0^T{\|\sigma^\top(s,X(s))(X(s)-x^* + \mu(s) M(X(s)))\|^2 ds}\right) \\ 
            \leq & \ \mathbb{E}\Bigg(\int_0^T 2(\|\sigma^\top(s,X(s))(X(s)-x^*)\|^2 + \mu_\text{up}^2\|\sigma^\top(s,X(s))(M(X(s))-M(x^*))\|^2) ds\Bigg)\\
            \leq & \ \mathbb{E}\Bigg(2\sup_{t\in[0,T]}{\|X(t)-x^*\|^2}\int_0^T{\sigma_\infty^2(s)  ds} + \mu_\text{up}^2 \sup_{t\in[0,T]}{\|M(X(t)) - M(x^*)\|^2}\int_0^T{\sigma_\infty^2(s) ds} \Bigg)\\
             \leq & \ 2(1+\mu_\text{up}^2L^2)\mathbb{E}\left( \sup_{t\in[0,T]}{\|X(t)-x^*\|^2}\right)\int_0^T{\sigma_\infty^2(s) ds}<+\infty.
        \end{align*}
By Proposition~\ref{prop:ito-formula}, this means that $N$ is a square-integrable continuous martingale, and $\mathbb{E}(N(t))=0$ for all $t \geq 0$. Taking the expectation of~\eqref{eq:ito-formula-anchor-function} and using that
        \begin{align*}
            \begin{split}
                & 0 \leq \text{tr}(\Sigma(s,X(s))\leq\sigma_\infty^2(s), \quad
                \langle M(X(s)), X(s)-x^*\rangle \geq 0, \quad \mbox{and} \quad  \|M(X(s))\|^2\geq 0 \quad \forall s \geq 0,
            \end{split}
        \end{align*}
yields for all $t \geq 0$
        \begin{align*}
            \mathbb{E}\left(\frac{1}{2}\|X(t)-x^*\|^2\right) \leq & \ \mathbb{E}(\phi(t,X(t)))\\
            = & \ \frac{1}{2}\|X_0-x^*\|^2 + \mu_\text{up}\langle M(X_0), X_0-x^*\rangle + \frac{\mu_\text{up}^2}{2}\|M(X_0)\|^2\\
            & \ - \underbrace{\mathbb{E}(U(t))}_{\geq 0} + \underbrace{\mathbb{E}(N(t))}_{=0} + \mathbb{E}(A(t))\\
            \leq & \ \frac{1}{2}\|X_0-x^*\|^2 +\mu(0) \langle M(X_0), X_0-x^*\rangle + \frac{\mu(0)^2}{2}\|M(X_0)\|^2 + \frac{1}{2}\int_0^{t}{\sigma_{\infty}^2(s) ds}<+\infty.
        \end{align*}
Finally, by taking the supremum over $t\geq 0$, we establish the validity of the first claim.

(ii) $A$ and $U$ are continuous adapted increasing processes with $A(0)=U(0)=0$ a.s., while $\lim_{t\rightarrow +\infty} A(t)$ $< +\infty$ holds due to the integrability condition for the diffusion term. Thus, we can use Theorem~\ref{thm:A.9-in-paper} to obtain almost surely that $\lim_{t \rightarrow +\infty} \phi(t,X(t),M(X(t)))$ exists and is finite, and
        \begin{align}\label{eq:defined-integrals-squared-norms-inner-product-tdp}
        \begin{split}
          \int_0^{+\infty} (\gamma(s) - \dot{\mu}(s))\left(\langle M(X(s)), X(s)-x^*\rangle +\mu(s)\|M(X(s))\|^2\right)  ds <+\infty.
        \end{split}
        \end{align}

In other words, for all $x^* \in \zer M$ there exists $\Omega_{x^*}\in\mathcal{F}$ such that $\mathbb{P}(\Omega_{x^*})=1$ and for all $\omega \in \Omega_{x^*}$ it holds that
        \begin{align}\label{eq:for-separability-argument-tdp}
 \lim_{t\rightarrow +\infty} \mu(t)\langle M(X(\omega, t)), X(\omega, t)-x^*\rangle + \frac{1}{2} \|X(\omega, t)-x^*\|^2 + \frac{\mu(t)^2}{2}\|M(X(\omega, t))\|^2 \quad \text{exists and is finite}.
        \end{align}
This implies that for all $x^* \in \zer M$ there exists $\Omega_{x^*}\in\mathcal{F}$ such that $\mathbb{P}(\Omega_{x^*})=1$ and $\|X(\omega, \cdot)-x^*\|$ remains bounded for all $\omega \in \Omega_{x^*}$. This proves that $\sup_{t\geq 0}{\|X(t)\|} < +\infty$ a.s.
       
(iii) Using the monotonicity of $M$ and the fact that $\mu$ and $\gamma$ are bounded from below by a positive constant, by~\eqref{eq:defined-integrals-squared-norms-inner-product-tdp}, it follows $\int_0^{+\infty}\|M(X(\omega,s))\|^2 ds < +\infty$ a.s. We denote by $\Omega_I$ the set of probability one for which it holds $\int_0^{+\infty}\|M(X(\omega,s))\|^2 ds < +\infty$ for all $\omega\in\Omega_I$.

 Let $R(t)=\int_0^t \sigma(s,X(s)) dW(s)$. This is a continuous martingale with respect to $\mathcal{F}_t$, which, according to the It\^o isometry property, satisfies
        \begin{align*}
            \mathbb{E}(\|R(t)\|^2) = \mathbb{E}\left(\int_0^t\|\sigma(s,X(s))\|_F^2 ds\right) \leq \mathbb{E}\left(\int_0^\infty \sigma_{\infty}^2(s) ds\right) < +\infty \quad \forall t\geq0.
        \end{align*}
By Theorem~\ref{thm:cont-martingale}, there exists a $\R^n$-valued random variable $R_\infty$ with respect to $\mathcal{F}_\infty$, the $\sigma$-algebra generated by $\cup_{t \geq 0} {\cal F}_t$, which satisfies $\mathbb{E}(\|R_\infty\|^2)<+\infty$, and
        \begin{align*}
            \lim_{t\rightarrow +\infty} R(\omega, t) = R_\infty(\omega) \quad \text{for every} \quad \omega\in\Omega_R,
        \end{align*}
where $\Omega_R\in\mathcal{F}$ is such that $\mathbb{P}(\Omega_R)=1$.

Then $\mathbb{P}(\Omega_I \cap \Omega_R)=1$ and, as we will see below, $\lim_{t\rightarrow +\infty}{\|M(X(\omega,t))\|}=0$ for all $\omega \in \Omega_I \cap \Omega_R$. Indeed, for all $\omega \in \Omega_I \cap \Omega_R$, since $\int_0^{+\infty}\|M(X(\omega, s))\|^2 ds < +\infty$, it holds $\liminf_{t\rightarrow +\infty}{\|M(X(\omega,t))\|}=0$. In addition, for all $\omega \in \Omega_I \cap \Omega_R$, $\limsup_{t\rightarrow +\infty}{\|M(X(\omega,t))\|}=0$. 

Suppose by contradiction that there exists $\omega_0 \in \Omega_I \cap \Omega_R$ such that $\limsup_{t\rightarrow +\infty}{\|M(X(\omega_0,t))\|}>0$.     Then, by Lemma~\ref{lem:limsup-liminf}, there exist $\delta>0$ satisfying
        \begin{align*}
            0=\liminf_{t\rightarrow +\infty}\|M(X(\omega_0,t))\|<\delta<\limsup_{t\rightarrow +\infty}{\|M(X(\omega_0, t))\|},
        \end{align*}
        and a sequence $(t_k)_{k\geq 0}\subseteq\R_+$ such that 
        \begin{align*}
           \lim_{k\rightarrow\infty} t_k=+\infty, \quad \|M(X(\omega, t_k))\|>\delta, \quad \text{and} \quad t_{k+1}-t_k>1 \quad \forall k \geq 0.
        \end{align*}

Let $\alpha >0$ such that $1 - (1+\alpha)L^2 \mu^2_\text{up} >0$ and $\varepsilon >0$ such that
$$\frac{\varepsilon}{2(|\dot{\mu}_{\text{low}}|^2 + \gamma_\text{up}^2)} < 1 \quad \mbox{and} \quad \left(3+ \frac{3}{\alpha}\right) \frac{\varepsilon L^2}{1-(1+\alpha)L^2\mu_\text{up}^2} < \frac{\delta^2}{2}.$$
Then $\left[t_k, t_k + \frac{\varepsilon}{2(|\dot{\mu}_{\text{low}}|^2 + \gamma_\text{up}^2)}\right ]$ are disjoint intervals for any two distinct $k \geq 0$. 

Also, by the convergence property of $R(\omega_0, \cdot)$ and the fact that $\int_0^{+\infty}\|M(X(\omega_0, s))\| ds <+\infty$, there exists $k_0 \geq 0$ such that for every $k\geq k_0$,
        \begin{align*}
            \sup_{t\geq t_k}{\|R(\omega_0,t)-R(\omega_0,t_k)\|^2}<\frac{\varepsilon}{2} \quad \text{and} \quad \int_{t_k}^{+\infty} \|M(X(\omega_0,s))\|^2 ds < 1.
        \end{align*}
        
For all $k\geq k_0$ and all $t\in\left[t_k, t_k + \frac{\varepsilon}{2(\dot{\mu}_{\text{low}}|^2 + \gamma_\text{up}^2)}\right ]$, it holds
        \begingroup
        \allowdisplaybreaks
        \begin{align*}
            & \ \|X(\omega_0, t) -X(\omega_0,t_k)\|^2\\
            = & \ \|X(\omega_0, t)+\mu(t)M(X(\omega_0, t)) - X(t_k)-\mu(t_k)M(X(\omega_0, t_k)) + \mu(t_k)M(X(\omega_0, t_k)) - \mu(t)M(X(\omega_0, t))\|^2 \\
            \leq & \ (1+\alpha) \|\mu(t_k)(M(X(\omega_0, t_k))-M(X(\omega_0, t)))\|^2 \\
            & \ + \left(1+ \frac{1}{\alpha}\right) \left \|\int_{t_k}^t (\dot{\mu}(s)-\gamma(s))M(X(\omega_0, s)) ds + \int_{t_k}^t \sigma(s,X(\omega_0, s)) dW(s) + (\mu(t_k) - \mu(t))M(X(\omega_0, t))\right \|^2 \\
            \leq & \ (1+\alpha) \mu_\text{up}^2 \|M(X(\omega_0, t_k))-M(X(\omega_0, t))\|^2\\
            & + \left(3+ \frac{3}{\alpha}\right) \left (\left \|\int_{t_k}^t (\dot{\mu}(s)-\gamma(s))M(X(\omega_0, s)) ds \right \|^2 + \left \|\int_{t_k}^t \sigma(s,X(\omega_0, s)) dW(s) \right \|^2 \right.\\
            & \qquad \qquad \qquad \quad + |\mu(t_k) - \mu(t)|^2\|M(X(\omega_0, t)) \|^2 \Bigg )\\
            \leq & \ (1+\alpha) \mu_\text{up}^2 \|M(X(\omega_0, t_k))-M(X(\omega_0, t))\|^2\\
            & + \left(3+ \frac{3}{\alpha}\right) \left (2(|\dot{\mu}_{\text{low}}|^2 + |\gamma_\text{up}|^2)(t-t_k)\int_{t_k}^t \|M(X(\omega_0, s))\|^2 ds + \left \|R(\omega_0,t) - R(\omega_0,t_k) \right \|^2  + \right.\\
           & \qquad \qquad \qquad + 4\mu_\text{up}^2\|M(X(\omega_0,t))\|^2 \Bigg )\\
            \leq & \ (1+\alpha) \mu_\text{up}^2 \|M(X(\omega_0, t_k))-M(X(\omega_0, t))\|^2 \\
            & + \left(3+ \frac{3}{\alpha}\right) \left (2(t-t_k)(|\dot{\mu}_{\text{low}}|^2 + \gamma_\text{up}^2) + \frac{\varepsilon}{2}  + 4\mu_\text{up}^2\|M(X(\omega_0,t))\|^2 \right )\\
             \leq & \ (1+\alpha) \mu_\text{up}^2 \|M(X(\omega_0, t_k))-M(X(\omega_0, t))\|^2 + \left(3+ \frac{3}{\alpha}\right) \left (\varepsilon  + 4\mu_\text{up}^2\|M(X(\omega_0,t))\|^2 \right ),
        \end{align*}
        \endgroup
where the first inequality is the one between the geometric and the arithmetic mean, the second inequality follows from the Cauchy-Schwarz inequality and the third inequality follows from the H\"older inequality for integrals.

By making use of the Lipschitz continuity of $M$, it yields
\begin{align*}
& \|M(X(\omega_0, t_k))-M(X(\omega_0, t))\|^2 \\
\leq & \ (1+\alpha) L^2 \mu_\text{up}^2 \|M(X(\omega_0, t_k))-M(X(\omega_0, t))\|^2 + \left(3+ \frac{3}{\alpha}\right) L^2 \left (\varepsilon  + 4\mu_\text{up}^2\|M(X(\omega_0,t))\|^2 \right), 
\end{align*}
consequently,        
\begin{align*}
            \|M(X(\omega_0, t_k))-M(X(\omega_0, t))\|^2 \leq \left(3+ \frac{3}{\alpha}\right) \frac{L^2}{1-(1+\alpha)L^2\mu_\text{up}^2}\left (\varepsilon  + 4\mu_\text{up}^2\|M(X(\omega_0,t))\|^2 \right).
        \end{align*}
Therefore, for all $k\geq k_0$ and all $t\in\left[t_k, t_k + \frac{\varepsilon}{2(|\dot{\mu}_{\text{low}}|^2 + \gamma_\text{up}^2)}\right ]$, from
  \begin{align*}
  \|M(X(\omega_0, t))\|^2 &\geq \frac{1}{2} \|M(X(\omega_0,t_k))\|^2 - \|M(X(\omega_0,t) - M(X(\omega_0,t_k))\|^2,
\end{align*}
we get
        \begin{align*}
    \left(1 + \left(3+ \frac{3}{\alpha}\right) \frac{4 \mu_\text{up}^2 L^2}{1-(1+\alpha)L^2\mu_\text{up}^2} \right) \|M(X(\omega_0, t))\|^2 \geq \frac{\delta^2}{2} -\left(3+ \frac{3}{\alpha}\right) \frac{\varepsilon L^2}{1-(1+\alpha)L^2\mu_\text{up}^2}.
        \end{align*}
Finally,
        \begin{align*}
            & \ \left(1 + \left(3+ \frac{3}{\alpha}\right) \frac{4 \mu_\text{up}^2 L^2}{1-(1+\alpha)L^2\mu_\text{up}^2} \right) \int_0^{+\infty} \|M(X(\omega_0, s))\|^2 ds\\
            \geq & \ \sum_{k\geq k_0}\int_{t_k}^{t_k + \frac{\varepsilon}{2(|\dot{\mu}_{\text{low}}|^2 + \gamma_\text{up}^2)}} \|M(X(\omega_0, s))\|^2 ds\\
            \geq & \ \sum_{k\geq k_0} \frac{\varepsilon}{2(|\dot{\mu}_{\text{low}}|^2 + \gamma_\text{up}^2)}\left( \frac{\delta^2}{2} -\left(3+ \frac{3}{\alpha}\right) \frac{\varepsilon L^2}{1-(1+\alpha)L^2\mu_\text{up}^2} \right)  =+\infty,
        \end{align*}
        which contradicts $\int_0^{+\infty}\|M(X(\omega_0, s))\|^2 ds <+\infty$. This means that for all $\omega \in \Omega_I \cap \Omega_R$
        \begin{align*}
           0 = \liminf_{t\rightarrow +\infty} \|M(X(\omega, t))\| = \limsup_{t\rightarrow +\infty} \|M(X(\omega, t))\| = \lim_{t\rightarrow +\infty} \|M(X(\omega, t))\| = 0.
        \end{align*}
Then, for all $x^* \in \zer M$ and all $\omega \in \Omega_{x^*} \cap \Omega_I \cap \Omega_R$, the Cauchy-Schwarz inequality gives
        \begin{align*}
            0 \leq \lim_{t\rightarrow +\infty} \langle M(X(\omega, t)), X(\omega, t)-x^*\rangle \leq \lim_{t\rightarrow +\infty} \underbrace{\|M(X(\omega, t))\|}_{\rightarrow 0} \underbrace{\|X(\omega, t)-x^*\|}_\text{is bounded}=0.
        \end{align*}
Using~\eqref{eq:for-separability-argument-tdp} and that $\mu(t)\geq \mu_\text{low}>0$ for all $t \geq 0$, it follows that $\lim_{t\rightarrow +\infty} \|X(\omega, t)-x^*\|$ exists and is finite.

Next, we will employ an argument from \cite{s.-fadili-attouch} to demonstrate the existence of a subset of $\Omega$ of probability one, which is independent of a previously chosen $x^* \in \zer M$, such that for all $\omega$ in this subset and all $x^* \in \zer M$ the limit $\lim_{t\rightarrow +\infty} \|X(\omega, t)-x^*\|$ exists and is finite.

Taking into account that $\zer M$ is closed, by the separability of $\R^n$ there exists a countable set $S\subseteq \R^n$ which is dense in  $\zer M$. Because $S$ is countable,
        \begin{align*}
            \mathbb{P}\left(\bigcap_{s\in S}{\Omega_s}\right)=1-\mathbb{P}\left(\bigcup_{s\in S} \Omega_s^c\right) \geq 1-\sum_{s\in S} \mathbb{P}(\Omega_s^c)=1.
        \end{align*}
Then  $\mathbb{P} \left(\bigcap_{s\in S}{\Omega_s} \cap \Omega_I \cap \Omega_R \right) =1$. Let $\omega \in \bigcap_{s\in S}{\Omega_s} \cap \Omega_I \cap \Omega_R$ be fixed and choose $x^*\in \zer M$. Then there exists a sequence $(s^k)_{k\geq 0}\subseteq S$ such that $s^k\rightarrow x^*$ as $k \rightarrow +\infty$. For every $k \geq 0$, since $\omega \in \Omega_{s^k}$,
        \begin{align*}
            \lim_{t\rightarrow +\infty}{\|X(\omega,t)-s^k\|} \ \mbox{exists and is finite}.
        \end{align*}
Applying the triangle inequality, it follows
        \begin{align*}
            | \|X(\omega,t)-s^k\| - \|X(\omega,t)-x^*\| |\leq \|s^k-x^*\| \quad \forall k \geq 0 \ \forall t \geq 0.
        \end{align*}
This gives for all $k \geq 0$
\begin{align*}
- \|s^k-x^*\| +  \lim_{t\rightarrow +\infty}{\|X(\omega,t)-s^k\|}
\leq &  \ \liminf_{t\rightarrow +\infty}{\|X(\omega,t)-x^*\|} \\
\leq & \ \limsup_{t\rightarrow +\infty}{\|X(\omega,t)-x^*\|} 
\leq  \lim_{t\rightarrow +\infty}{\|X(\omega,t)-s^k\|} + \|s^k-x^*\|.
\end{align*}
Letting $k \rightarrow +\infty$, we get 
$$\lim_{t\rightarrow +\infty} \|X(\omega,t)-x^*\| = \lim_{k \rightarrow +\infty} \lim_{t\rightarrow +\infty}{\|X(\omega,t)-s^k\|} \in \R.$$

Finally, we recall that there exists $\Omega_\text{cont} \in {\cal F}$ such that $\mathbb{P}(\Omega_\text{cont})=1$ and $X(\omega,\cdot)$ is continuous for every $\omega\in\Omega_\text{cont}$. Let $\omega \in \Omega_\text{converge}:= \bigcap_{s\in S}{\Omega_s} \cap \Omega_I \cap \Omega_R \cap \Omega_\text{cont}$. It holds $\mathbb{P}(\Omega_\text{converge})=1$. For $x^* \in \zer M$, there exists $C(\omega) \in \R$ and $T(\omega) \leq 0$ such that $\|X(\omega,t)-x^*\|\leq C(\omega)$ for all $t\geq T(\omega)$. Because $X(\omega,\cdot)$ is continuous, it holds
        \begin{align*}
            \sup_{t\in[0,T(\omega)]}{\|X(\omega,t)\|} < +\infty.
        \end{align*}
        Thus,
        \begin{align*}
            \sup_{t\geq 0} \|X(\omega, t)\| \leq \max\left (\sup_{t\in[0,T(\omega)]}{\|X(\omega,t)\|}, C(\omega)+\|x^*\| \right ) < +\infty.
        \end{align*}
        
(iv) We will use Opial's Lemma to prove the convergence of the trajectory process. We fix $\omega \in \Omega_\text{converge}$ and recall that above we proved that for every $x^* \in \zer M$ the limit $\lim_{t\rightarrow +\infty} \|X(\omega,t)-x^*\|$ exists. In addition, $X(\omega, \cdot)$ is bounded, therefore, its set of limit points is not empty. Let $\overline x(\omega)$ be such a limit point, which means that there exists a sequence $(t_k)_{k \geq 0}\subseteq \R_+$ such that \begin{align*}
            \lim_{k\rightarrow\infty}{X(\omega,t_k)}=\overline{x}(\omega).
        \end{align*}
From $\lim_{t\rightarrow +\infty}{\|M(X(\omega, t))\|}=0$ and the continuity of $M$, we see that $\overline{x}(\omega)\in \zer M$. Since both conditions in Opial's Lemma are satisfied, there exists $x^*(\omega) \in \zer M$ such that $\lim_{t\rightarrow +\infty} X(\omega,t)=x^*(\omega)$. 

Since $\omega\in\Omega_\text{converge}$ was arbitrarily chosen, there exists a $\zer M$-valued random variable $x^*$ such that $\lim_{t\rightarrow +\infty} X(t) = x^*$ almost surely.
\end{proof}

\section{Continuous time system: convergence rates in expectation}\label{sec-cont-time-rates}

In this section we will provide upper bounds and convergence rates in expectation for the ergodic squared norm of the operator and the ergodic gap function. In addition, in case $M$ is $\kappa$-strongly monotone with constant $\kappa >0$, we derive upper bounds in expectation for the squared distance of the trajectory process to the unique zero of $M$. 

We recall that $M : \R^n \rightarrow \R^n$ is called $\kappa$-strongly monotone with constant $\kappa >0$ if  $\langle Mx - My, x - y\rangle \geq \kappa \|x-y\|^2$ for all $x,y \in \R^n$.

Regarding the parameter functions, throughout this section, we will assume that:
\begin{align}\label{condmugamma2}
 0 < \mu(t) \leq \mu_\text{up}:=\mu(0) \quad \forall t\geq 0 \ \mbox{and} \ \mu \ \mbox{is nonincreasing on} \ [0,+\infty).
\end{align}

\begin{thm}\label{thm:rates-tdp}
Let $X$  be a trajectory process of~\eqref{eq:SDE-M} and $x^*\in\ \zer M$. Then the following statements are true:
    \begin{enumerate}
    \item[(i)] Assume that $\gamma$ is nonincreasing on $[0,+\infty)$. For all $t > 0$ it holds
        \begin{align*}
            \mathbb{E}\left(\frac{1}{t} \int_0^t \langle X(s)-x^*,M(X(s))\rangle ds\right) \leq \frac{1}{\gamma(t) t}\left(\frac{\mu_\text{up}^2L^2}{2}+ \mu_\text{up} L +\frac{1}{2}\right)\distance(X_0,\zer M)^2+\frac{\sigma_*^2}{2\gamma(t)\mu(t)}
        \end{align*}
        and
        \begin{align*}
            \mathbb{E}\left(\frac{1}{t}\int_0^t \|M(X(s))\|^2 ds\right)\leq \frac{1}{\gamma(t) \mu(t)t}\left(\frac{\mu_\text{up}^2L^2}{2}+\mu_\text{up} L +\frac{1}{2}\right)\distance(X_0,\zer M)^2+\frac{\sigma_*^2}{2\gamma(t)\mu(t)}.
        \end{align*}
        
        If, in addition, $\int_0^{+\infty}\sigma_\infty^2(s) ds < +\infty$, then
        \begin{align*}
            \mathbb{E}\left(\frac{1}{t}\int_0^t \langle X(s)-x^*,M(X(s))\rangle ds\right)=\mathcal{O}\left(\frac{1}{\gamma(t)t}\right) \quad \mbox{as} \ t \rightarrow +\infty.
        \end{align*}
        and
         \begin{align*}
            \mathbb{E}\left(\frac{1}{t}\int_0^t \|M(X(s))\|^2 ds\right)=\mathcal{O}\left(\frac{1}{\gamma(t)\mu(t)t}\right) \quad \mbox{as} \ t \rightarrow +\infty
        \end{align*}
        \item[(ii)] Assume that $0 < \gamma_\text{low} \leq \gamma(t)$ for all $t \geq 0$. For all $t > 0$ it holds
        \begin{align*}
            \mathbb{E}\left(\frac{1}{t} \int_0^t \langle X(s)-x^*,M(X(s))\rangle ds\right) \leq \frac{1}{\gamma_\text{low} t}\left(\frac{\mu_\text{up}^2L^2}{2}+ \mu_\text{up} L +\frac{1}{2}\right)\distance(X_0,\zer M)^2+\frac{\sigma_*^2}{2\gamma_\text{low}\mu(t)}
        \end{align*}
        and
        \begin{align*}
            \mathbb{E}\left(\frac{1}{t}\int_0^t \|M(X(s))\|^2 ds\right)\leq \frac{1}{\gamma_\text{low} \mu(t)t}\left(\frac{\mu_\text{up}^2L^2}{2}+\mu_\text{up} L +\frac{1}{2}\right)\distance(X_0,\zer M)^2+\frac{\sigma_*^2}{2\gamma_\text{low}\mu(t)}.
        \end{align*}
        
        If, in addition, $\int_0^{+\infty}\sigma_\infty^2(s) ds < +\infty$, then
        \begin{align*}
            \mathbb{E}\left(\frac{1}{t}\int_0^t \langle X(s)-x^*,M(X(s))\rangle ds\right)=\mathcal{O}\left(\frac{1}{t}\right) \quad \mbox{as} \ t \rightarrow +\infty.
        \end{align*}
        and
         \begin{align*}
            \mathbb{E}\left(\frac{1}{t}\int_0^t \|M(X(s))\|^2 ds\right)=\mathcal{O}\left(\frac{1}{\mu(t)t}\right) \quad \mbox{as} \ t \rightarrow +\infty
        \end{align*}
        \item[(iii)] Assume that $0 < \gamma_\text{low} \leq \gamma(t)$ for all $t \geq 0$. If $M$ is $\kappa$-strongly monotone with constant $\kappa \geq\frac{1}{2\mu_\text{up}}$ and $X_0 \neq x^*$, then for all $t \geq 0$ it holds
        \begin{align*}
            \mathbb{E}\left(\frac{1}{2}\|X(t)-x^*\|^2\right)\leq \left(\frac{1}{2}\|X_0-x^*\|^2 +\mu_\text{up}\langle M(X_0),X_0-x^*\rangle + \frac{\mu_\text{up}^2}{2}\|M(X_0)\|^2\right)e^{-\frac{\gamma_\text{low}}{2\mu_\text{up}}t}+\frac{\sigma_*^2\mu_\text{up}}{\gamma_\text{low}}.
        \end{align*}
        If, in addition, $\sigma_\infty$ is decreasing and vanishes at $+\infty$, then for all $\lambda\in(0,1)$ and all $t > 0$ it holds
        \begin{align*}
            \mathbb{E}\left(\frac{1}{2}\|X(t)-x^*\|^2 \right) \leq & \ \left(\frac{1}{2}\|X_0-x^*\|^2 +\mu_\text{up}\langle M(X_0),X_0-x^*\rangle + \frac{\mu_\text{up}^2}{2}\|M(X_0)\|^2\right)e^{-\frac{\gamma_\text{low}}{2\mu_\text{up}}t}\\
            & \ +\frac{\sigma_*^2\mu_\text{up}}{\gamma_\text{low}}e^{-\frac{\gamma_\text{low}}{2\mu_\text{up}}(1-\lambda)t}+\frac{\mu_\text{up}}{\gamma_\text{low}}\sigma_\infty^2(\lambda t).
        \end{align*}
    \end{enumerate}
\end{thm}
\begin{proof}
(i) For all $t \geq 0$ we set (see \eqref{eq:ito-formula-anchor-function})
        \begin{align*}
        \phi(t,X(t),M(X(t)))
            & = \mu(0)\langle M(X_0),X_0-x^*\rangle + \frac{1}{2}\|X_0-x^*\|^2 + \frac{\mu(0)^2}{2}\|M(X_0)\|^2 
            \\&\quad- \int_0^t (-\dot{\mu}(s) + \gamma(s)) (\langle M(X(s)), X(s)-x^*\rangle + \mu(s)\|M(X(s))\|^2) ds
            \\&\quad+ \int_0^t \langle \sigma^\top(s,X(s))(\mu(s)M(X(s))+X(s)-x^*),dW(s)\rangle
            \\&\quad+ \frac{1}{2}\int_0^t \text{tr}(\Sigma(s,X(s))) ds
        \end{align*}
        and $G(t):=\mathbb{E}(\phi(t,X(t),M(X(t))))$. Then, using that the expectation of the stochastic integral is equal to zero and that $\dot \mu(t) \leq 0$, for all $t \geq 0$ it yields
        \begin{align}\label{eqG}
            G(t)-G(0) = & -\mathbb{E}\left(\int_0^t (-\dot{\mu}(s) + \gamma(s)) (\langle M(X(s)), X(s)-x^*\rangle + \mu(s)\|M(X(s))\|^2) ds\right) \nonumber\\
            & + \frac{1}{2}\mathbb{E}\left(\int_0^t \text{tr}(\Sigma(s,X(s)))ds\right)\\
            \leq & -\gamma(t) \mathbb{E}\left(\int_0^t \langle M(X(s)), X(s)-x^*\rangle ds\right) -\gamma(t) \mu(t) \mathbb{E}\left(\int_0^t \|M(X(s))\|^2 ds\right) + \frac{\sigma_*^2t}{2}.\nonumber
        \end{align}
Rearranging and using that $G(t)\geq 0$, we obtain for all $t >0$
        \begin{align*}
        \begin{split}
            & \gamma(t) \mathbb{E}\left(\frac{1}{t}\int_0^t \langle M(X(s)), X(s)-x^*\rangle ds\right) + \gamma(t) \mu(t)\mathbb{E}\left(\frac{1}{t}\int_0^t \|M(X(s))\|^2 ds\right)\\
            \leq & \ \frac{1}{t}\left(\mu(0)\langle M(X_0),X_0-x^*\rangle +\frac{1}{2}\|X_0-x^*\|^2 + \frac{\mu(0)^2}{2}\|M(X_0)\|^2\right)+\frac{\sigma_*^2}{2}\\
            \leq & \ \frac{1}{t}\left(\mu(0) L +\frac{1}{2}+\frac{\mu(0)^2 L^2}{2}\right)\|X_0-x^*\|^2+\frac{\sigma_*^2}{2}.
        \end{split}
        \end{align*}
Taking the infimum over all $x^*\in \zer M$ yields the result.

If $\int_0^{+\infty}\sigma_\infty^2(s) ds < +\infty$, then we have for all $t > 0$
\begin{align*}
             & \ G(t)-G(0)\\
            \leq & \ -\gamma(t) \mathbb{E}\left(\int_0^t \langle M(X(s)), X(s)-x^*\rangle ds\right) -\gamma(t)\mu(t) \mathbb{E}\left(\int_0^t \|M(X(s))\|^2 ds\right) + \frac{1}{2}\mathbb{E}\left(\int_0^{+\infty} \sigma_\infty^2(s) ds\right),
        \end{align*}
which, after rearrangement, gives
\begin{align*}
        \begin{split}
            & \gamma(t) \mathbb{E}\left(\frac{1}{t}\int_0^t \langle M(X(s)), X(s)-x^*\rangle ds\right) + \gamma(t)\mu(t)\mathbb{E}\left(\frac{1}{t}\int_0^t \|M(X(s))\|^2 ds\right)\\
            \leq & \ \frac{1}{t}\left(\mu(0)\langle M(X_0),X_0-x^*\rangle +\frac{1}{2}\|X_0-x^*\|^2 + \frac{\mu(0)^2}{2}\|M(X_0)\|^2\right)+ \frac{1}{2t}\mathbb{E}\left(\int_0^{+\infty} \sigma_\infty^2(s) ds\right),
        \end{split}
        \end{align*}
thus proving
     \begin{align*}
            \mathbb{E}\left(\frac{1}{t}\int_0^t \langle X(s)-x^*,M(X(s))\rangle ds\right)=\mathcal{O}\left(\frac{1}{\gamma(t)t}\right) \ \mbox{as} \ t \rightarrow +\infty
        \end{align*}
        and
        \begin{align*}
            \mathbb{E}\left(\frac{1}{t}\int_0^t\|M(X(s))\|^2 ds\right)=\mathcal{O}\left(\frac{1}{\gamma(t)\mu(t)t}\right) \ \mbox{as} \ t \rightarrow +\infty.
        \end{align*}   
        
(ii) The proof follows in the same lines as that of statement (i).

(iii) By using the strong monotonicity of the operator, for all $t_2 > t_1 \geq 0$ it holds
        \begin{align*}
            & \ \mathbb{E}\left(\int_{t_1}^{t_2} (\dot{\mu}(s) - \gamma(s)) (\langle M(X(s)), X(s)-x^*\rangle + \mu(s)\|M(X(s))\|^2) ds\right) \\
            \leq & \ \mathbb{E}\left(\int_{t_1}^{t_2} \left(-\mu(s)\gamma(s)\|M(X(s))\|^2 - \gamma(s)\langle X(s)-x^*, M(X(s))\rangle \right ) ds\right)\\
               \leq & \ \gamma_\text{low} \mathbb{E}\left(\int_0^t \left(-\mu(s) \|M(X(s))\|^2 - \langle X(s)-x^*, M(X(s))\rangle \right)ds\right)\\
            \leq & \ \frac{\gamma_\text{low}}{2\mu_\text{up}} \mathbb{E}\left(\int_{t_1}^{t_2} \left(-2\mu(s) \mu_\text{up} \|M(X(s))\|^2 - 2\mu_\text{up} \langle X(s)-x^*, M(X(s))\rangle \right) ds\right)\\
            \leq & \ \frac{\gamma_\text{low}}{2\mu_\text{up}} \mathbb{E}\left(\int_{t_1}^{t_2} \left(-2\mu(s)^2 \|M(X(s))\|^2 - \mu(s) \langle X(s)-x^*, M(X(s))\rangle - \mu_{up} \kappa \|X(s)-x^*\|^2 \right)ds\right)\\
            \leq & \ \frac{\gamma_\text{low}}{2\mu_\text{up}} \mathbb{E}\left(\int_{t_1}^{t_2} \left(-\frac{\mu(s)^2}{2}\|M(X(s))\|^2 - \mu(s)\langle X(s)-x^*,M(X(s))\rangle - \frac{1}{2}\|X(s)-x^*\|^2 \right) ds\right)\\
            = & \ -\frac{\gamma_\text{low}}{2\mu_\text{up}}\int_{t_1}^{t_2} G(s) ds.
        \end{align*}
Hence, according to \eqref{eq:ito-formula-anchor-function}, for all $t_2 > t_1 \geq 0$ it holds
        \begin{align*}
            G(t_2) \leq G(t_1) -\frac{\gamma_\text{low}}{2\mu_\text{up}}\int_{t_1}^{t_2}  G(s) ds + \frac{1}{2} \int_{t_1}^{t_2} \sigma_\infty^2(s)ds \leq G(t_1) -\frac{\gamma_\text{low}}{2\mu_\text{up}}\int_{t_1}^{t_2}  G(s) ds + \frac{1}{2}\int_{t_1}^{t_2} \sigma_*^2ds.
        \end{align*}
The solution of the ordinary differential equation
        \begin{align*}
            \begin{cases}
                & \varphi'(t)=-\frac{\gamma_\text{low}}{2\mu_\text{up}}\varphi(t)+\frac{\sigma_*^2}{2} \quad \forall t > 0
                \\& \varphi(0)= G(0)
            \end{cases}
        \end{align*}
is
    \begin{align*}
        \varphi(t) = G(0) e^{-\frac{\gamma_\text{low}}{2\mu_\text{up}}t}+\frac{\sigma_*^2 \mu_\text{up}}{\gamma_\text{low}} (1-e^{-\frac{\gamma_\text{low}}{2\mu_\text{up}}t}) \quad \forall t\geq 0,
    \end{align*}
hence, invoking Lemma~\ref{lem:estimate-solution-cauchy-problem}, it holds
    \begin{align*}
        \mathbb{E}\left(\frac{1}{2}\|X(t)-x^*\|^2 \right) \leq G(t)\leq G(0) e^{-\frac{\gamma_\text{low}}{2\mu_\text{up}}t}+\frac{\sigma_*^2 \mu_\text{up}}{\gamma_\text{low}} \quad \forall t\geq 0,
    \end{align*}
and the claim follows.

For $\sigma_\infty$ decreasing and vanishing at $+\infty$, we consider the ordinary differential equations
    \begin{align*}
        \begin{cases}
            & \varphi'(t) = -\frac{\gamma_\text{low}}{2\mu_\text{up}}\varphi(t) + \frac{\sigma_\infty^2(t)}{2} \quad \forall t > 0
            \\& \varphi(0) = G(0).
        \end{cases}
    \end{align*}
Employing a technique used in \cite{s.-fadili-attouch}, we choose $\lambda\in(0,1)$ and derive for its solution the following estimates, that hold for all $t \geq 0$
    \begin{align*}
        \varphi(t) &= G(0)e^{-\frac{\gamma_\text{low}}{2\mu_\text{up}}t} + e^{-\frac{\gamma_\text{low}}{2\mu_\text{up}}t}\int_0^t \frac{\sigma_\infty^2(s)}{2}e^{\frac{\gamma_\text{low}}{2\mu_\text{up}}s} ds
        \\&\leq G(0)e^{-\frac{\gamma_\text{low}}{2\mu_\text{up}}t} + e^{-\frac{\gamma_\text{low}}{2\mu_\text{up}}t}\left(\int_0^{\lambda t}\frac{\sigma_\infty^2(s)}{2}e^{\frac{\gamma_\text{low}}{2\mu_\text{up}}s}ds + \int_{\lambda t}^t \frac{\sigma_\infty^2(s)}{2}e^{\frac{\gamma_\text{low}}{2\mu_\text{up}}s} ds \right)
        \\&\leq G(0)e^{-\frac{\gamma_\text{low}}{2\mu_\text{up}t}} + e^{-\frac{\gamma_\text{low}}{2\mu_\text{up}}t}\left(\frac{\sigma_*^2}{2} \int_0^{\lambda t}e^{\frac{\gamma_\text{low}}{2\mu_\text{up}}s}ds+\frac{\sigma_\infty^2(\lambda t)}{2}\int_{\lambda t}^t e^{\frac{\gamma_\text{low}}{2\mu_\text{up}}s} ds\right)
        \\& \leq G(0)e^{-\frac{\gamma_\text{low}}{2\mu_\text{up}}t} + e^{-\frac{\gamma_\text{low}}{2\mu_\text{up}}t}\left(\frac{\sigma_*^2 \mu_\text{up}}{\gamma_\text{low}}e^{\frac{\gamma_\text{low}}{2\mu_\text{up}}\lambda t} + \frac{2\mu_\text{up}}{\gamma_\text{low}}\frac{\sigma_\infty^2(\lambda t)}{2}e^{\frac{\gamma_\text{low}}{2\mu_\text{up}}t}\right).
    \end{align*}
Hence, invoking again Lemma~\ref{lem:estimate-solution-cauchy-problem}, it holds
    \begin{align*}
        G(t)\leq G(0)e^{-\frac{\gamma_\text{low}}{2\mu_\text{up}}t} + \frac{\sigma_*^2 \mu_\text{up}}{\gamma_\text{low}}e^{-\frac{\gamma_\text{low}}{2\mu_\text{up}}(1-\lambda)t} + \frac{\mu_\text{up}}{\gamma_\text{low}}\sigma_\infty^2(\lambda t) \quad \forall t\geq 0,
    \end{align*}
proving the desired result.
\end{proof}

\begin{rmk}\label{rem32}
\rm    While the boundedness of $\gamma$ from above is essential in the proof of Theorem~\ref{thm:convergence-tdp}, the convergence rates in Theorem~\ref{thm:rates-tdp} hold also for $\gamma$ not bounded from above.
\end{rmk}

\begin{rmk}\label{rem33}
\rm Finding the saddle points of the minimax problem \eqref{minimax}
\begin{align*}
    \min_{x\in\R^p}\max_{y\in\R^q} \Phi(x,y),
\end{align*}
where $\Phi : \R^p \times \R^q \rightarrow \R$ is a convex-concave function with $L$-Lipschitz continuous gradient, is equivalent to solving equation \eqref{eq:find-zeros} for
\begin{align*}
    M: \R^p \times \R^q \rightarrow \R, \quad M(x,y):=(\nabla_x \Phi(x,y), -\nabla_y \Phi(x,y)).
\end{align*}
The corresponding stochastic differential equation \eqref{eq:SDE-M} reads
\begin{align}
\tag{SDE-minimax}
\label{eq:SDE-minimax}
\begin{split}
    \begin{cases} &d(X(t)+\mu(t)\nabla_x \Phi(X(\cdot),Y(\cdot))) = - (\gamma(t) - \dot{\mu}(t)) \nabla_x \Phi(X(\cdot),Y(\cdot)) dt + \sigma_1(t,X(t), Y(t)) dW(t),\\
    &d(Y(t)-\mu(t)\nabla_y \Phi(X(\cdot),Y(\cdot))) = (\gamma(t) - \dot{\mu}(t)) \nabla_y \Phi(X(\cdot),Y(\cdot)) dt + \sigma_2(t,X(t), Y(t)) dW(t) \quad \forall t > 0,\\
    & X(0)=X_0, Y(0)=Y_0,
    \end{cases}
\end{split}
\end{align}
and is defined over a filtered probability space $(\Omega,\mathcal{F},\{\mathcal{F}_t\}_{t\geq 0},\mathbb{P})$ with $\sigma_1: \R_+ \times \R^p \times \R^q \rightarrow \R^{p \times m}$ and $\sigma_2 : \R_+ \times \R^p \times \R^q \rightarrow \R^{q \times m}$ measurable such that $\begin{pmatrix} \sigma_1 \\ \sigma_2 \end{pmatrix} : \R_+ \times \R^p \times \R^q \rightarrow \R^{(p+q) \times m}$ satisfies \eqref{eq:assumption-sigma}, $W$ a $m$-dimensional Brownian motion, $X(\cdot)$, $Y(\cdot)$,  $\nabla_x \Phi(X(\cdot),Y(\cdot))$, and $\nabla_y \Phi(X(\cdot),Y(\cdot))$ are stochastic It\^{o} processes with the same $m$-dimensional Brownian motion $W$, and parameter functions $\mu : [0,+\infty) \rightarrow (0,+\infty)$ and $\gamma : [0,+\infty) \rightarrow (0,+\infty)$ assumed to be continuous differentiable and, respectively, integrable and to fulfill \eqref{condmugamma}.

The system \eqref{eq:SDE-minimax} has a unique trajectory solution $(X,Y) \in S_{p+q}^\nu$, for all $\nu \geq 2$. The trajectory $(X,Y)$ satisfies all the statements of Theorem \ref{thm:convergence-tdp} under the corresponding hypotheses. Furthermore, for the gap function it holds almost surely $\lim_{t\rightarrow +\infty}{\Phi(X(t),y^*)-\Phi(x^*,Y(t))}=0$ for any saddle point $(x^*,y^*)$ of \eqref{minimax}. This is a straightforward consequence of the fact that for all $t \geq 0$
    \begin{align}\label{eq:for-minimax-conv}
    \begin{split}
  & \ \|(\nabla_x \Phi(X(t),Y(t)), -\nabla_y \Phi(X(t),Y(t)))\|   \|(X(t),Y(t))-(x^*,y^*)\|  \\
  \geq & \ \langle (\nabla_x \Phi(X(t),Y(t)), -\nabla_y \Phi(X(t),Y(t))), (X(t),Y(t))-(x^*,y^*)\rangle
        \\ = & \ \langle \nabla_x \Phi(X(t),Y(t)), X(t)-x^*\rangle + \langle -\nabla_y \Phi(X(t),Y(t)), Y(t)-y^*\rangle
        \\ \geq & \ \Phi(X(t),Y(t)) - \Phi(x^*,Y(t)) - \Phi(X(t),Y(t)) + \Phi(X(t),y^*)
        \\ = & \ \Phi(X(t),y^*) - \Phi(x^*,Y(t)) \geq 0.
    \end{split}
    \end{align}

The trajectory $(X,Y)$ satisfies also all the statements of Theorem \ref{thm:rates-tdp} under the corresponding hypotheses. Furthermore, for any saddle point $(x^*,y^*)$ of \eqref{minimax} it holds
 \begin{align*}
           & \  \mathbb{E}\left(\Phi\left(\frac{1}{t}\int_0^t X(s)ds,y^*\right)-\Phi\left(x^*,\frac{1}{t} \int_0^t Y(s)ds\right)\right) \\
            \leq & \ \frac{1}{\gamma(t) t}\left(\frac{\mu_\text{up}^2L^2}{2}+ \mu_\text{up} L +\frac{1}{2}\right)\distance((X_0,Y_0), \zer M)^2+\frac{\sigma_*^2}{2\gamma(t)\mu(t)} \quad \forall t > 0,
        \end{align*}
if $\gamma$ is nonincreasing on $[0,+\infty)$, and
 \begin{align*}
            & \ \mathbb{E}\left(\Phi\left(\frac{1}{t}\int_0^t X(s)ds,y^*\right)-\Phi\left(x^*,\frac{1}{t} \int_0^t Y(s)ds\right)\right) \\
            \leq & \ \frac{1}{\gamma_\text{low} t}\left(\frac{\mu_\text{up}^2L^2}{2}+ \mu_\text{up} L +\frac{1}{2}\right)\distance((X_0,Y_0), \zer M)^2+\frac{\sigma_*^2}{2\gamma_\text{low}\mu(t)} \quad \forall t > 0,
        \end{align*}
if $0 < \gamma_\text{low} \leq \gamma(t)$ for all $t \geq 0$. If, in addition, $\int_0^{+\infty}\sigma_\infty^2(s) ds < +\infty$, then
        \begin{align*}
            \mathbb{E}\left(\Phi\left(\frac{1}{t}\int_0^t X(s)ds,y^*\right)-\Phi\left(x^*,\frac{1}{t} \int_0^t Y(s)ds\right)\right) =\mathcal{O}\left(\frac{1}{\gamma(t)t}\right) \quad \mbox{as} \ t \rightarrow +\infty
        \end{align*}
and
\begin{align*}
            \mathbb{E}\left(\Phi\left(\frac{1}{t}\int_0^t X(s)ds,y^*\right)-\Phi\left(x^*,\frac{1}{t} \int_0^t Y(s)ds\right)\right) =\mathcal{O}\left(\frac{1}{\gamma(t)t}\right) \quad \mbox{as} \ t \rightarrow +\infty,
        \end{align*}
respectively.

Indeed, invoking Jensen's inequality and \eqref{eq:for-minimax-conv}, we have for all $t > 0$
\begin{align*}
0 \leq & \ \mathbb{E}\left(\Phi\left(\frac{1}{t}\int_0^t X(s)ds,y^*\right)-\Phi\left(x^*,\frac{1}{t} \int_0^t Y(s)ds\right)\right)
\leq  \mathbb{E}\left(\frac{1}{t} \int_0^t  (\Phi(X(s),y^*)-\Phi(x^*,Y(s))) ds \right) \\
\leq & \ \mathbb{E}\left(\frac{1}{t} \int_0^t \left \langle (\nabla_x \Phi(X(s),Y(s)), -\nabla_y \Phi(X(s),Y(s))), (X(s),Y(s))-(x^*,y^*) \right \rangle ds \right).
\end{align*}
\end{rmk}

\section{Discrete time considerations}\label{sec-disc-time}

The aim of this section is to show how the ergodic upper bound results in expectation for the squared norm of the operator and for the gap function transfer to the stochastic variants of the Optimistic Gradient Descent Ascent (OGDA) method and the Extragradient (EG) method with constant step sizes, which can be interpreted as temporal discretizations of ~\eqref{eq:SDE-M}. 

We will consider only equations governed by operators which are monotone, but not necessarily strongly monotone. For a stochastic variant of the Optimistic Gradient Descent Ascent (OGDA) method designed to solve variational inequalities, ergodic upper bound results for the gap function in expectation have also been derived in \cite{boehm-sedlmayer-csetnek-bot}. For the stochastic Extragradient (EG) method, ergodic upper bound results in expectation for the squared norm of the operator have also been derived in \cite{gbgl} (see also \cite{bmsv, ijot}), and for the gap function \cite{boehm-sedlmayer-csetnek-bot}. As seen in Remark \ref{rem33}, the upper bound results for the gap function can be straightforwardly transferred via Jensen's inequality and the gradient inequality for convex functions to the ergodic primal-dual gap when solving convex-concave minimax problems.

Throughout this section, we will assume that only a stochastic estimator $M(\cdot, \xi)$ is accessible instead of $M$ itself. The stochastic estimator is assumed to be unbiased, meaning that $\mathbb{E}_\xi(M(x,\xi))=M(x)$ for all $x\in\R^n$. Further, it is required to have bounded variance, i.e., $\mathbb{E}_\xi(\|M(x,\xi)-M(x)\|^2)<\sigma_*^2$ for all $x\in\R^n$. 

\subsection{Stochastic Optimistic Gradient Descent Ascent (OGDA) method}\label{subsec-sogda}

Recalling the stochastic differential equation ~\eqref{eq:SDE-M}
\begin{align*}
    d(X(t)+\mu(t)M(X(t))) = - (\gamma(t) - \dot{\mu}(t)) M(X(t)) dt + \sigma(t,X(t)) dW(t),
\end{align*}
and fixing the parameter functions $\mu(t)=\gamma(t):=\gamma>0$ to be constant for all $t \geq 0$, we consider the temporal discretization
\begin{align*}
    x^{k+1} - x^k +\gamma M(x^k,\xi_k) - \gamma M(x^{k-1},\xi_{k-1}) = -\gamma M(x^k,\xi_k) \quad \forall k \geq 1.
\end{align*}
Equivalently,
\begin{align}\tag{Stochastic OGDA}\label{eq:SOGDA}
    x^{k+1} := x^k - 2\gamma M(x^k, \xi_k) + \gamma M(x^{k-1}, \xi_{k-1}) \quad \forall k \geq 1,
\end{align}
where $x^0, x^1 \in \R^n$ are given vectors and $\xi_k$ encodes the stochasticity appearing at iteration $k$. 

We consider the approximation errors
$$W_k := M(x^k, \xi_k) - M(x^k) \quad \forall k \geq 0,$$
and the sub-sigma algebras
$${\cal F}_k := \sigma(x^0, x^1, \xi_0, \xi_1, ..., \xi_{k-1}) \quad \forall k \geq 2.$$

\begin{lem}\label{lem:for-average-iterate-rate-sogda}
Let $x^* \in \zer M$ and $\gamma < \frac{1}{2L}$. For all $k \geq 2$ it holds
    \begin{align*}
        \|x^{k+1} + \gamma M(x^k,\xi_k)-x^*\|\leq & \ \|x^k + \gamma M(x^{k-1},\xi_{k-1}) - x^*\|^2\\
        & - 2\gamma\langle M(x^k,\xi_k), x^k-x^*\rangle + \frac{\gamma^2(16\gamma^2L^2-1)}{2(1-4\gamma^2L^2)} \|M(x^{k-1})\|^2 \nonumber\\
        & + \frac{12\gamma^4L^4}{1-4\gamma^2L^2}\left(\|x^{k-1} - x^{k-2}\|^2 -   \|x^k - x^{k-1}\|^2 \right)\\
& +  3 \gamma^2 \|W_k\|^2 + \frac{\gamma^2(16\gamma^2L^2+5)}{2(1-4\gamma^2L^2)} \|W_{k-1}\|^2 + \frac{12\gamma^4L^2}{1-4\gamma^2L^2} \|W_{k-2}\|^2.
    \end{align*}
\end{lem}
\begin{proof} Let $k \geq 2$ be fixed. We observe that
\begingroup
        \allowdisplaybreaks
    \begin{align}\label{eq1alg1}
        & \ \|x^{k+1} + \gamma M(x^k,\xi_k)-x^*\|^2 \nonumber \\
        = & \ \|x^k - \gamma M(x^k,\xi_k) +  \gamma M(x^{k-1},\xi_{k-1}) -x^*\|^2 \nonumber \\
         = & \ \|x^k - \gamma M(x^k,\xi_k) +  2\gamma M(x^{k-1},\xi_{k-1}) -x^*\|^2  - \gamma^2\|M(x^{k-1},\xi_{k-1})\|^2 \nonumber\\
         & \ - 2\gamma \langle x^k - \gamma M(x^k,\xi_k) +  \gamma M(x^{k-1},\xi_{k-1}) -x^*, M(x^{k-1},\xi_{k-1}) \rangle \nonumber \\
        = & \ \|x^k + \gamma M(x^{k-1},\xi_{k-1}) - x^*\|^2 + \|\gamma M(x^{k-1}, \xi_{k-1}) - \gamma M(x^k,\xi_k)\|^2 - \gamma^2\|M(x^{k-1}, \xi_{k-1})\|^2 \nonumber\\
        & \ + 2 \gamma \langle x^k + \gamma M(x^{k-1},\xi_{k-1}) - x^*, M(x^{k-1}, \xi_{k-1}) - M(x^k,\xi_k)\rangle \nonumber\\
        & \ - 2\gamma \langle x^k - \gamma M(x^k,\xi_k) +  \gamma M(x^{k-1},\xi_{k-1}) -x^*, M(x^{k-1},\xi_{k-1}) \rangle \nonumber\\
        = & \ \|x^k + \gamma M(x^{k-1},\xi_{k-1}) - x^*\|^2 + \gamma^2(\|M(x^{k-1}, \xi_{k-1}) - M(x^k, \xi_k)\|^2 - \|M(x^{k-1},\xi_{k-1})\|^2) \nonumber\\
        & \ -2\langle \gamma M(x^k,\xi_k), x^k - x^*\rangle \nonumber \\
        \leq & \ \|x^k + \gamma M(x^{k-1},\xi_{k-1}) - x^*\|^2 - 2\gamma\langle M(x^k,\xi_k), x^k-x^*\rangle - \gamma^2 \|M(x^{k-1},\xi_{k-1})\|^2 \nonumber \\
        &\ + \gamma^2\big(\|W_{k-1} + M(x^{k-1} - M(x^k))+ W_k\|^2\big) \nonumber\\
        \leq & \ \|x^k + \gamma M(x^{k-1},\xi_{k-1}) - x^*\|^2 - 2\gamma\langle M(x^k,\xi_k), x^k-x^*\rangle - \gamma^2 \|M(x^{k-1},\xi_{k-1})\|^2 \nonumber \\
        &+ 3\gamma^2(\|W_{k-1}\|^2  + \|M(x^{k-1})-M(x^k)\|^2 + \|W_k\|^2). \nonumber\\
        \leq & \ \|x^k + \gamma M(x^{k-1},\xi_{k-1}) - x^*\|^2 - 2\gamma\langle M(x^k,\xi_k), x^k-x^*\rangle - \gamma^2 \|M(x^{k-1},\xi_{k-1})\|^2 \nonumber\\
        &+ 3\gamma^2(\|W_{k-1}\|^2 + L^2\|x^{k-1}-x^k\|^2 + \|W_k\|^2).
    \end{align}
    \endgroup
On the other hand,
    \begin{align*}
        \|x^k - x^{k-1}\|^2 & \ = \gamma^2 \|-2M(x^{k-1},\xi_{k-1}) + M(x^{k-2},\xi_{k-2}) \|^2\\ 
        & \ = \gamma^2 \|-M(x^{k-1},\xi_{k-1}) + W_{k-1} + M(x^{k-2}) - M(x^{k-1}) + W_{k-2} \|^2\\ 
        & \ \leq 4 \gamma^2 (\|M(x^{k-1},\xi_{k-1})\|^2 + \|W_{k-1}\|^2 + L^2\|x^{k-1} - x^{k-2}\|^2 + \|W_{k-2}\|^2),
    \end{align*}
therefore,
    \begin{align}\label{eq1alg2}
    (1-4\gamma^2L^2)  \|x^k - x^{k-1}\|^2  \leq & \ 4\gamma^2 (\|M(x^{k-1},\xi_{k-1})\|^2 + \|W_{k-1}\|^2  + \|W_{k-2}\|^2) \nonumber \\
    & \ + 4 \gamma ^2L^2 \left(\|x^{k-1} - x^{k-2}\|^2 -   \|x^k - x^{k-1}\|^2 \right).
    \end{align}
Plugging \eqref{eq1alg2} into \eqref{eq1alg1}, we obtain
    \begin{align*}
       & \ \|x^{k+1} + \gamma M(x^k,\xi_k)-x^*\|^2\\
        \leq & \ \|x^k + \gamma M(x^{k-1},\xi_{k-1}) - x^*\|^2 - 2\gamma\langle M(x^k,\xi_k), x^k-x^*\rangle - \gamma^2 \|M(x^{k-1},\xi_{k-1})\|^2 \nonumber\\
& + \frac{12\gamma^4L^2}{1-4\gamma^2L^2} \|M(x^{k-1},\xi_{k-1})\|^2 + 3 \gamma^2 \|W_k\|^2 + \left(3\gamma^2 + \frac{12\gamma^4L^2}{1-4\gamma^2L^2}\right) \|W_{k-1}\|^2 + \frac{12\gamma^4L^2}{1-4\gamma^2L^2} \|W_{k-2}\|^2\\ 
& + \frac{12\gamma^4L^4}{1-4\gamma^2L^2}\left(\|x^{k-1} - x^{k-2}\|^2 -   \|x^k - x^{k-1}\|^2 \right)\\
= & \ \|x^k + \gamma M(x^{k-1},\xi_{k-1}) - x^*\|^2 - 2\gamma\langle M(x^k,\xi_k), x^k-x^*\rangle + \frac{16\gamma^4L^2-\gamma^2}{1-4\gamma^2L^2} \|M(x^{k-1},\xi_{k-1})\|^2 \nonumber\\
&  + 3 \gamma^2 \|W_k\|^2 + \left(3\gamma^2 + \frac{12\gamma^4L^2}{1-4\gamma^2L^2}\right) \|W_{k-1}\|^2 + \frac{12\gamma^4L^2}{1-4\gamma^2L^2} \|W_{k-2}\|^2\\ 
& + \frac{12\gamma^4L^4}{1-4\gamma^2L^2}\left(\|x^{k-1} - x^{k-2}\|^2 -   \|x^k - x^{k-1}\|^2 \right)\\
\leq & \ \|x^k + \gamma M(x^{k-1},\xi_{k-1}) - x^*\|^2 - 2\gamma\langle M(x^k,\xi_k), x^k-x^*\rangle + \frac{\gamma^2(16\gamma^2L^2-1)}{2(1-4\gamma^2L^2)} \|M(x^{k-1})\|^2 \nonumber\\
& +  3 \gamma^2 \|W_k\|^2 + \frac{\gamma^2(16\gamma^2L^2+5)}{2(1-4\gamma^2L^2)} \|W_{k-1}\|^2 + \frac{12\gamma^4L^2}{1-4\gamma^2L^2} \|W_{k-2}\|^2\\ 
& + \frac{12\gamma^4L^4}{1-4\gamma^2L^2}\left(\|x^{k-1} - x^{k-2}\|^2 -   \|x^k - x^{k-1}\|^2 \right),
\end{align*}
as claimed.
\end{proof}

The following theorem provides ergodic upper bounds in expectation for the squared norm of the operator and the gap function for \eqref{eq:SOGDA}.

\begin{thm}\label{thm:average-iterate-rate-sogda}
Let $x^* \in \zer M$ and $\gamma < \frac{1}{4L}$. For all $K \geq 0$ it holds
    \begin{align*}
 & \ \min_{k=1, ..., K+1}  \mathbb{E}\left( \|M(x^k)\|^2 \right) \leq  \mathbb{E}\left(\frac{1}{K+1}\sum_{k=1}^{K+1}{\|M(x^k)\|^2} \right) \\
\leq & \ \frac{1}{K+1} \frac{2(1-4\gamma^2L^2)}{\gamma^2(1-16\gamma^2L^2)}\left(4\|x^1-x^*\|^2  + \frac{4 \gamma^2 L^2(1-\gamma^2 L^2)}{1- 4 \gamma^2L^2} \|x^{1} - x^{0}\|^2 + 8\gamma^2\sigma_*^2 \right) + \frac{16\gamma^2L^2+11}{1-16\gamma^2L^2} \sigma_*^2
    \end{align*}
    and
    \begin{align*}
        \mathbb{E}\left(\frac{1}{K+1} \sum_{k=2}^{K+2} \langle M(x^k), x^k-x^*\rangle \right) \leq & \ \frac{1}{K+1} \left(\frac{2\|x^1-x^*\|^2}{\gamma}  + \frac{2 \gamma L^2(1-\gamma^2 L^2)}{1- 4 \gamma^2L^2} \|x^{1} - x^{0}\|^2 + 4 \gamma \sigma_*^2 \right)\\
        & \ + \frac{\gamma(16\gamma^2L^2+11)}{4(1-4\gamma^2L^2)} \sigma_*^2.
    \end{align*}
\end{thm}
\begin{proof}
Let $K \geq 0$. According to Lemma~\ref{lem:for-average-iterate-rate-sogda}, we have for all $k \geq 2$
    \begin{align*}
       & \frac{\gamma^2(1-16\gamma^2L^2)}{2(1-4\gamma^2L^2)} \|M(x^{k-1})\|^2 + 2\gamma\langle M(x^k,\xi_k), x^k-x^*\rangle\\
       \leq & \ \|x^k + \gamma M(x^{k-1},\xi_{k-1})-x^*\|^2 - \|x^{k+1} + \gamma M(x^k,\xi_k)-x^*\|^2\\
& \  + \frac{12\gamma^4L^4}{1-4\gamma^2L^2}\left(\|x^{k-1} - x^{k-2}\|^2 -   \|x^k - x^{k-1}\|^2 \right)\\
         & \ + \frac{12\gamma^4L^2}{1-4\gamma^2L^2} \|W_{k-2}\|^2 + \frac{\gamma^2(16\gamma^2L^2+5)}{2(1-4\gamma^2L^2)} \|W_{k-1}\|^2 + 3 \gamma^2 \|W_k\|^2.
    \end{align*}
Summing the above inequality for $k$ from $2$ to $K+2$ and multiplying by $\frac{1}{K+1}$ yields
    \begin{align*}
       & \frac{\gamma^2(1-16\gamma^2L^2)}{2(1-4\gamma^2L^2)} \frac{1}{K+1} \sum_{k=1}^{K+1} \|M(x^{k})\|^2 + \frac{2\gamma}{K+1} \sum_{k=2}^{K+2} \langle M(x^k,\xi_k), x^k-x^*\rangle\\
       \leq & \ \frac{1}{K+1} \left (\|x^2 + \gamma M(x^{1},\xi_{1})-x^*\|^2 - \|x^{K+3} + \gamma M(x^{K+2},\xi_{K+2})-x^*\|^2 \right)\\
       & \ + \frac{1}{K+1} \frac{12\gamma^4L^4}{1-4\gamma^2L^2}\left(\|x^{1} - x^{0}\|^2 -   \|x^{K+2} - x^{K+1}\|^2 \right)\\
         & \ + \frac{1}{K+1} \sum_{k=2}^{K+2} \left(\frac{12\gamma^4L^2}{1-4\gamma^2L^2} \|W_{k-2}\|^2 + \frac{\gamma^2(16\gamma^2L^2+5)}{2(1-4\gamma^2L^2)} \|W_{k-1}\|^2 + 3 \gamma^2 \|W_k\|^2\right)\\
         \leq & \ \frac{1}{K+1} \left(\|x^2 + \gamma M(x^{1},\xi_{1})-x^*\|^2 + \frac{12\gamma^4L^4}{1-4\gamma^2L^2} \|x^{1} - x^{0}\|^2 \right) \\
         & \ + \frac{1}{K+1} \sum_{k=2}^{K+2} \left(\frac{12\gamma^4L^2}{1-4\gamma^2L^2} \|W_{k-2}\|^2 + \frac{\gamma^2(16\gamma^2L^2+5)}{2(1-4\gamma^2L^2)} \|W_{k-1}\|^2 + 3 \gamma^2 \|W_k\|^2\right).
    \end{align*}
For all $k = 2, ..., K+2$ we have that 
$$\mathbb{E}(\langle M(x^k,\xi_k), x^k-x^*\rangle) = \mathbb{E}(\mathbb{E}(\langle M(x^k,\xi_k), x^k-x^*\rangle | {\cal F}_{k})) =  \mathbb{E}( \langle \mathbb{E}(M(x^k,\xi_k) | {\cal F}_{k}), x^k-x^*\rangle) = \langle M(x^k), x^k-x^*\rangle $$
and 
$$\mathbb{E}(\|W_k\|^2) = \mathbb{E}(\mathbb{E}(\|W_k\|^2 | {\cal F}_{k})) \leq \sigma_*^2.$$
This yields
    \begingroup
        \allowdisplaybreaks
    \begin{align*}
       & \frac{\gamma^2(1-16\gamma^2L^2)}{2(1-4\gamma^2L^2)} \mathbb{E}\left(\frac{1}{K+1}\sum_{k=1}^{K+1}{\|M(x^k)\|^2} \right) + 2\gamma \mathbb{E}\left(\frac{1}{K+1} \sum_{k=2}^{K+2} \langle M(x^k), x^k-x^*\rangle \right)\\
       \leq & \ \frac{1}{K+1} \left(\mathbb{E}(\|x^2 + \gamma M(x^{1},\xi_{1})-x^*\|^2) + \frac{12\gamma^4L^4}{1-4\gamma^2L^2} \|x^{1} - x^{0}\|^2 \right) \\
         & \ + \frac{1}{K+1} \sum_{k=2}^{K+2} \mathbb{E}\left(\frac{12\gamma^4L^2}{1-4\gamma^2L^2} \|W_{k-2}\|^2 + \frac{\gamma^2(16\gamma^2L^2+5)}{2(1-4\gamma^2L^2)} \|W_{k-1}\|^2 + 3 \gamma^2 \|W_k\|^2\right)\\
       \leq & \ \frac{1}{K+1} \left(4\|x^1-x^*\|^2 + 4\gamma^2L^2\|x^1-x^0\|^2 + 8\gamma^2\sigma_*^2 + \frac{12\gamma^4L^4}{1-4\gamma^2L^2} \|x^{1} - x^{0}\|^2 \right) + \frac{\gamma^2(16\gamma^2L^2+11)}{2(1-4\gamma^2L^2)} \sigma_*^2\\
       = & \ \frac{1}{K+1} \left(4\|x^1-x^*\|^2  + \frac{4 \gamma^2 L^2(1-\gamma^2 L^2)}{1- 4 \gamma^2L^2} \|x^{1} - x^{0}\|^2 + 8\gamma^2\sigma_*^2 \right) + \frac{\gamma^2(16\gamma^2L^2+11)}{2(1-4\gamma^2L^2)} \sigma_*^2.
    \end{align*}
    \endgroup
This concludes the proof.
\end{proof}

\subsection{Stochastic Extragradient (EG) method}\label{subsec-seg}

The stochastic Extragradient method (EG) reads
\begin{align}\tag{Stochastic EG}\label{eq:SEG}
\begin{split}
    \begin{cases}
        & y^k := x^k - \gamma M(x^k, \xi_k) \\
        & x^{k+1} := x^k - \gamma M(y^k, \eta_k)
    \end{cases} \quad \forall k \geq 0,
\end{split}
\end{align}
where $x^0 \in \R^n$ is a given vector and $\xi_k$ and $\eta_k$ encode the stochasticity appearing at iteration $k$ when evaluating $M$ at $x^k$ and $y^k$, respectively.

We consider the approximation errors
$$W_k := M(x^k, \xi_k) - M(x^k) \quad \mbox{and} \quad Z_k := M(y^k, \eta_k) - M(y^k)  \quad \forall k \geq 0$$
and the sub-sigma algebras
$${\cal F}_0 := \sigma(x^0) \quad \mbox{and} \quad {\cal F}_k := \sigma(x^0,\xi_0, \xi_1, ..., \xi_{k-1}, \eta_0, ..., \eta_{k-1}) \quad \forall k \geq 1.$$
and 
$$\widehat{\cal F}_0 := \sigma(x^0, \xi_0) \quad \mbox{and} \quad 
 \widehat{\cal F}_k := \sigma(x^0,\xi_0, \xi_1, ..., \xi_{k-1}, \xi_k, \eta_0, ..., \eta_{k-1}) \quad \forall k \geq 1.$$

First, establish an inequality that will be essential to the proof of the convergence rate:
\begin{lem}\label{lem:for-average-iterate-rate}
    Let $x^* \in \zer M$. For all $k \geq 0$ it holds
    \begin{align*}
        \|x^{k+1}-x^*\|^2 \leq & \ \|x^k-x^*\|^2-2\gamma\langle M(y^k,\eta_k), y^k-x^*\rangle + 2(3L^2\gamma^2-1)\|M(x^k)\|^2\\
         & \ + \gamma^2 (3\|Z_k\|^2+(1+6L^2\gamma^2)\|W_k\|^2).
    \end{align*}
\end{lem}
\begin{proof} Let $k \geq 0$. We have
\begingroup
        \allowdisplaybreaks
    \begin{align*}
        \|x^{k+1}-x^*\|^2 = & \ \|x^k - \gamma M(y^k,\eta_k)-x^*\|^2
         \\ = & \ \|x^k-x^*\|^2 - 2\gamma \langle x^k - x^*, M(y^k,\eta_k) - M(x^*) \rangle + \gamma^2 \|M(y^k,\eta_k) - M(x^*)\|^2
         \\ = & \ \|x^k-x^*\|^2 - 2\gamma \langle y^k + \gamma M(x^k,\xi_k) - \gamma M(x^*) - x^*, M(y^k,\eta_k) - M(x^*) \rangle
         \\& + \gamma^2 \|M(y^k,\eta_k) - M(x^*)\|^2
         \\= & \ \|x^k-x^*\|^2 - 2\gamma \langle y^k-x^*, M(y^k,\eta_k)-M(x^*)\rangle
         \\& \ - \gamma^2 \left (2\langle M(x^k,\xi_k)-M(x^*), M(y^k,\eta_k)-M(x^*)\rangle - \|M(y^k,\eta_k) - M(x^*)\|^2 \right )
         \\= & \ \|x^k-x^*\|^2 - 2\gamma\langle y^k-x^*, M(y^k,\eta_k)-M(x^*)\rangle - \gamma^2 \|M(x^k,\xi_k)\|^2\\
         & \ + \gamma^2(\|M(y^k,\eta_k)-M(x^k,\xi_k)\|^2)\\
         \leq & \  \|x^k-x^*\|^2 - 2\gamma\langle y^k-x^*,M(y^k,\eta_k)-M(x^*)\rangle - \gamma^2 \|M(x^k,\xi_k)\|^2\\
         & + \gamma^2(3\|M(y^k,\eta_k)-M(y^k)\|^2 + 3\|M(y^k)-M(x^k)\|^2
         + 3\|M(x^k)-M(x^k,\xi_k)\|^2)\\
          \leq & \ \|x^k-x^*\|^2 - 2\gamma \langle y^k-x^*, M(y^k,\eta_k)-M(x^*)\rangle + \gamma^2(3L^2-1) \|M(x^k,\xi_k)\|^2\\
         & + \gamma^2(3\|Z_k\|^2 + 3\|W_k\|^2)\\
         \leq & \ \|x^k-x^*\|^2-2\gamma\langle y^k-x^*,M(y^k,\eta_k)-M(x^*)\rangle + 2(3L^2\gamma^2-1)\|M(x^k)\|^2\\
         & \ + \gamma^2 (3\|Z_k\|^2 +(1+6L^2\gamma^2)\|W_k\|^2),
    \end{align*}
    \endgroup
    as claimed.
\end{proof}

The following theorem provides ergodic upper bounds in expectation for the squared norm of the operator in terms of the sequence $(x^k)_{k \geq 0}$ and the gap function  in terms of the sequence $(y^k)_{k \geq 0}$.

\begin{thm}\label{thm:average-iterate-rate}
Let $x^* \in \zer M$ and $\gamma < \frac{1}{\sqrt{3}L}$. For all $K \geq 0$ it holds
    \begin{align*}
\min_{k=0, ..., K} \mathbb{E} \left(\|M(x^k)\|^2\right) \leq \mathbb{E} \left(\frac{1}{K+1}\sum_{k=0}^{K}{\|M(x^k)\|^2}\right) \leq \frac{1}{K+1} \frac{1}{2(1-3L^2\gamma^2)} \|x^0-x^*\|^2 + \frac{\gamma^2(2 + 3L^2\gamma^2)}{1-3L^2\gamma^2}\sigma_*^2
    \end{align*}
and
    \begin{align*}
\mathbb{E}\left(\frac{1}{K+1}\sum_{k=0}^{K} \langle M(y^k),y^k-x^*\rangle \right) 
\leq  \frac{1}{K+1} \frac{1}{2\gamma} \|x^0-x^*\|^2 + \gamma(2 + 3L^2\gamma^2)\sigma_*^2.
    \end{align*}
\end{thm}
\begin{proof}
Let $K \geq 0$. According to Lemma~\ref{lem:for-average-iterate-rate}, we have for all $k \geq 0$
    \begin{align*}
        2(1-3L^2\gamma^2)\|M(x^k)\|^2 + 2\gamma\langle M(y^k,\eta_k), y^k-x^*\rangle \leq & \ \|x^k-x^*\|^2 - \|x^{k+1}-x^*\|^2\\
        & + \gamma^2 (3\|Z_k\|^2+(1+6L^2\gamma^2)\|W_k\|^2).        
    \end{align*}
Summing the above inequality for $k$ from $0$ to $K$ and multiplying by $\frac{1}{K+1}$ it yields
    \begin{align*}
        & \ \frac{2(1-3L^2\gamma^2)}{K+1}\sum_{k=0}^{K}\|M(x^k)\|^2 + \frac{2\gamma}{K+1}\sum_{k=0}^{K}{\langle M(y^k,\eta_k), y^k-x^*\rangle}\\
        \leq & \ \frac{1}{K+1}(\|x^0-x^*\|^2-\|x^{K+1}-x^*\|^2) + \frac{\gamma^2}{K+1}\left(\sum_{k=0}^{K}3\|Z_k\|^2+(1+6L^2\gamma^2)\|W_k\|^2\right)\\
        \leq & \ \frac{1}{K+1} \|x^0-x^*\|^2 + \frac{\gamma^2}{K+1}\left(\sum_{k=0}^{K}3\|Z_k\|^2+(1+6L^2\gamma^2)\|W_k\|^2\right).
    \end{align*}
For all $k = 0, ..., K$ we have that 
$$\mathbb{E}(\langle M(y^k,\eta_k), y^k-x^*\rangle) = \mathbb{E}(\mathbb{E}(\langle M(y^k,\eta_k), y^k-x^*\rangle | \widehat{\cal F}_{k})) =  \mathbb{E}( \langle \mathbb{E}(M(y^k,\eta_k) | \widehat{\cal F}_{k}), y^k-x^*\rangle) = \langle M(y^k), y^k-x^*\rangle,$$
and 
$$\mathbb{E}(\|W_k\|^2) = \mathbb{E}(\mathbb{E}(\|W_k\|^2 | {\cal F}_{k})) \leq \sigma_*^2 \quad \mbox{and} \quad \mathbb{E}(\|Z_k\|^2) = \mathbb{E}(\mathbb{E}(\|Z_k\|^2 | \widehat{\cal F}_{k})) \leq \sigma_*^2.$$
This yields
    \begin{align*}
       & \ 2(1-3L^2\gamma^2) \mathbb{E} \left(\frac{1}{K+1}\sum_{k=0}^{K}{\|M(x^k)\|^2}\right) + 2\gamma \mathbb{E} \left(\frac{1}{K+1}\sum_{k=0}^{K}{\langle M(y^k), y^k-x^*\rangle}\right)\\
        \leq & \ \frac{1}{K+1} \|x^0-x^*\|^2 + 2\gamma^2(2 + 3L^2\gamma^2)\sigma_*^2.
    \end{align*}
This concludes the proof.   
\end{proof}

\subsection{Numerical experiments}\label{numexp}

In order to illustrate the convergence behaviour of the two stochastic algorithms, we consider the monotone equation associated to the following minimax problem (see also ~\cite{bot-csetnek-nguyen, ouyang-xu})
\begin{align*}
    \min_{x\in\R^n} \max_{y\in\R^n} \Phi (x,y), 
\end{align*}
where
\begin{align*}
    \Phi (x,y) := \frac{1}{2}\langle x, Hx \rangle - \langle x,h\rangle - \langle y,Ax-b \rangle,
\end{align*}
with
\begin{align*}
    A := \frac{1}{4} \begin{pmatrix}
        & & & -1 & 1 \\
        & & \reflectbox{$\ddots$} & \reflectbox{$\ddots$} &\\
        & -1 & 1 & \\
        -1 & 1 & & \\
        1 & & & &
    \end{pmatrix}\in\R^{n\times n},\quad H := 2A^\top A,\quad
    b:= \frac{1}{4}\begin{pmatrix}
        1\\
        1\\
        \vdots\\
        1\\
        1
    \end{pmatrix} \in \R^n \text{ and }
    h := \frac{1}{4}\begin{pmatrix}
        0\\
        0\\
        \vdots\\
        0\\
        1
    \end{pmatrix} \in \R^n.
\end{align*}
For $n=10$, we performed the stochastic OGDA and EG methods with $50000$ iterations for a total of $100$ times. In Figure \eqref{fig:numerical-experiments} we represent the averaged squared norm of the operator 
$$M(x,y) = (\nabla_x \Phi(x,y), -\nabla_y \Phi(x,y))$$ 
for the stochastic OGDA and EG methods.The shaded areas represent the range between the best and worst-case instances, while the strongly colored lines depict the average over all $100$ runs of the algorithms. The error term in the $k$-th iteration was chosen as $\frac{a}{k\sqrt{n}} (1, ..., 1)^T$, where $a$ is normally distributed with mean $0$ and standard deviation $10$.

\begin{figure}[H]
    \centering
    \includegraphics[width=0.8\textwidth]{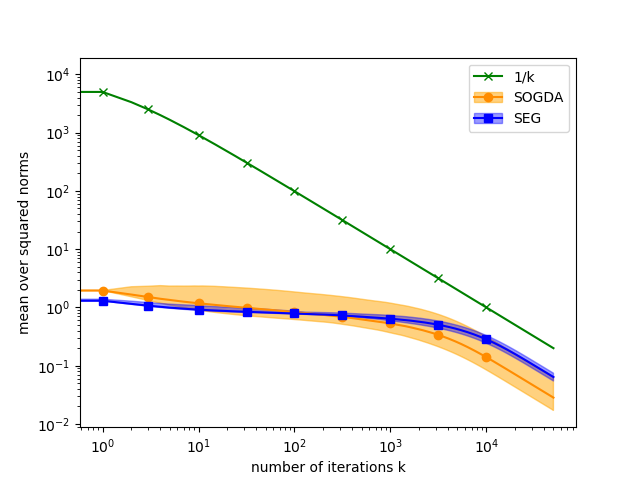}
    \caption{The convergence behaviour of the averaged squared norm of the operator for the stochastic OGDA and EG methods}
    \label{fig:numerical-experiments}
\end{figure}

\renewcommand\thesection{\Alph{section}}
\renewcommand\thesubsection{\thesection.\Alph{subsection}}
\setcounter{section}{0}

\section{Appendix}

In the appendix, we present several auxiliary results utilized in the analysis conducted in this paper. The following theorem will play a crucial role in establishing the existence and uniqueness of a solution to \eqref{eq:SDE-M}.

\begin{thm}(\cite[Theorem A.7]{s.-fadili-attouch}, \cite[Theorem 5.2.1]{oksendal})\label{thm:help-existence-uniqueness}
    Let $F:\R_+\times\R^n \rightarrow \R^n$ and $G:\R_+\times\R^n \rightarrow \R^{n \times m}$ be measurable functions satisfying, for every $T>0$
    \begin{align*}
        \|F(t,x)-F(t,y)\| + \|G(t,x)-G(t,y)\|_F \leq C_1 \|x-y\|, \quad \forall x,y\in\R^n \ \forall t\in[0,T],
    \end{align*}
    for some constant $C_1\geq 0$. Then the stochastic differential equation
    \begin{align*}
    \tag{SDE-gen}
    \label{SDE-gen}
        \begin{cases}
            & dX(t)= F(t,X(t)) dt + G(t,X(t)) dW(t) \quad \forall t\in [0,T],
            \\& X(0)=X_0,
        \end{cases}
    \end{align*}
    where $W$ is an $\mathcal{F}_t$-adapted $m$-dimensional Brownian motion, has a unique solution $X\in S_{n}^{\nu}[0,T]$, for every $\nu\geq 2$.
\end{thm}

Next, we introduce It\^{o}'s formula (see~\cite{evans}).
\begin{prop}\label{prop:ito-formula}
Let $X \in S^\nu_n$ be the solution of \eqref{SDE-gen}
    \begin{align*}
        \begin{cases}
            & dX(t)= F(t,X(t)) dt + G(t,X(t)) dW(t) \quad \forall t\geq 0,
            \\& X(0)=X_0,
        \end{cases}
    \end{align*}
    where $W$ is an $\mathcal{F}_t$-adapted $m$-dimensional Brownian motion, and, for all $T>0$, $F=(F_1, ..., F_n): \R_+\times\R^n \rightarrow \R^n$ satisfies
    \begin{align*}
        \mathbb{E}\left( \int_0^T |F_i(t,X(t))| dt \right) < +\infty,
    \end{align*}
    and $G=(G_{ij})_{i=\{1,\ldots,n\},j=\{1,\ldots,n\}} :\R_+\times\R^n \rightarrow \R^{n \times m}$ satisfies
    \begin{align*}
        \mathbb{E}\left( \int_0^T |G_{ij}(t,X(t))|^2 dt \right) < +\infty.
    \end{align*}
Further, let $\phi:\R_+\times\R^n \rightarrow \R$ be such that $\phi(\cdot,x)\in\text{C}^1(\R_+)$ for all $x\in\R^n$ and $\phi(t,\cdot)\in\text{C}^2(\R^n)$ for all $t\geq 0$. Then the process
    \begin{align*}
        \Tilde{X}(t)=\phi(t,X(t))
    \end{align*}
    is an It\^{o} process such that
    \begin{align*}
        d\Tilde{X}(t) = \frac{d}{dt} \phi(t,X(t)) dt &+ \sum_{i=1}^{n} \frac{d}{dx_i} \phi(t,X(t)) dX^i(t)
        \\&+ \frac{1}{2} \sum_{i,j=1}^{n} \frac{d^2}{dx_i dx_j} \phi(t,X(t))\sum_{\ell=1}^{m} G_{i\ell}(t,X(t)) G_{j\ell}(t,X(t)) dt \quad \forall t \geq 0.
    \end{align*}
    Further, if, for all $T>0$,
    \begin{align*}
        \mathbb{E}\left(\int_0^T \|\sigma^\top(s,(X(s),Y(s)))\nabla_x\phi(s,X(s))\|^2 ds\right)<+\infty,
    \end{align*}
    then $\int_0^t\left\langle \sigma^\top(s,(X(s),Y(s)))\nabla_x\phi(s,(X(s),Y(s))), dW(s) \right\rangle$ is for all $t \geq 0$ a square-integrable continuous martingale with expected value 0.
\end{prop}

The proof of the convergence results in Theorem~\ref{thm:convergence-tdp} relies on the following result.
\begin{thm}(\cite[Theorem 3.9]{mao}, Theorem 3.9)\label{thm:A.9-in-paper}
    Let $\{A(t)\}_{t\geq 0}$ and $\{U(t)\}_{t\geq 0}$ be two continuous adapted increasing processes with $A(0)=U(0)=0$ a.s. Let $\{N(t)\}_{t\geq 0}$ be a real-valued continuous local martingale with $N(0)=0$ a.s. Let $\xi$ be a nonnegative $\mathcal{F}_0$-measurable random variable. Define
    \begin{align*}
        X(t):=\xi + A(t) - U(t) + N(t) \quad \text{for} \quad \forall t\geq 0.
    \end{align*}
    If $X(t)$ is nonnegative and $\lim_{t\rightarrow +\infty}{A(t)} < +\infty$ a.s., then a.s. $\lim_{t\rightarrow +\infty}{X(t)}$ exists and is finite, as well as $\lim_{t\rightarrow +\infty}{U(t)} < +\infty$.
\end{thm}

In order to show that $\lim_{t\rightarrow+\infty}{\|M(X(t))\|} = 0$ a.s., we make use of the following lemma as well as of Doob's martingale convergence theorem presented below.

\begin{lem}(\cite[Lemma A.4]{s.-fadili-attouch})\label{lem:limsup-liminf}
    Let $f:\R_+\rightarrow\R$ be such that $\liminf_{t\rightarrow +\infty}f(t)\neq\limsup_{t\rightarrow +\infty}f(t)$. Then, there exists a constant $\alpha$ such that $\liminf_{t\rightarrow +\infty}f(t)<\alpha<\limsup_{t\rightarrow +\infty} f(t)$, and for every $\beta>0$ we can define a sequence $(t_k)_{k\in\N}\subseteq \R$ with the properties that
    \begin{align*}
        f(t_k)>\alpha \ \mbox{and} \ t_{k+1}>t_k+\beta,\quad\forall k\in\N.
    \end{align*}
\end{lem}

\begin{thm}(\cite{doob})\label{thm:cont-martingale}
    Let $\{M(t)\}_{t\geq 0} : \Omega \rightarrow \R$ be a continuous martingale such that $\sup_{t\geq 0}{\mathbb{E}(|M(t)|^p)}<+\infty$ for some $p>1$. Then there exists a random variable $M_\infty$ such that $\mathbb{E}(|M_\infty|^p)<+\infty$ and $\lim_{t\rightarrow +\infty} M(t) = M_\infty$ a.s.
\end{thm}

The derivation of convergence rate in the strongly monotone case necessitates the following lemma.
\begin{lem} (\cite[Lemma A.2]{s.-fadili-attouch})\label{lem:estimate-solution-cauchy-problem}
    Let $t_0\geq 0$ and $T>t_0$. Assume that $h:[t_0,+\infty) \rightarrow \R_+$ is measurable with $h\in\text{L}^1([t_0,T))$, that $\psi:\R_+\rightarrow\R_+$ is continuous and nondecreasing, and the Cauchy problem
    \begin{align*}
        \begin{cases}
            & \varphi'(t)=-\psi(\varphi(t))+h(t)\quad\text{for almost all}\quad t\in[t_0,T]
            \\& \varphi(t_0)=\varphi_0 >0
        \end{cases}
    \end{align*}
    has an absolutely continuous solution $\varphi:[t_0,T]\rightarrow\R_+$. If a bounded lower semicontinuous function $\omega:[t_0,T]\rightarrow\R_+$ satisfies $\omega(t_0)=\varphi_0$ and
    \begin{align*}
        \omega(t_2)\leq\omega(t_1)-\int_{t_1}^{t_2} \psi(\omega(s)) ds + \int_{t_1}^{t_2} h(s) ds \quad \forall t_0\leq t_1<t_2\leq T,
    \end{align*}
then
    \begin{align*}
        \omega(t)\leq \varphi(t) \quad \forall t\in[t_0,T].
    \end{align*}
\end{lem}

\nocite{*}
\printbibliography[]

\end{document}